\documentclass[12pt]{article}
\usepackage[utf8]{inputenc}
\usepackage{amsmath}
\usepackage{amsfonts}
\usepackage{mathrsfs}
\usepackage{mathtools}
\usepackage[dvipdfmx]{graphicx}
\usepackage[top=40truemm,bottom=40truemm,left=30truemm,right=30truemm]{geometry}
\usepackage{here}
\usepackage{amsthm}
\usepackage{algorithmic}
\usepackage{algorithm}
\usepackage{url}
\usepackage{stmaryrd}
\usepackage[title]{appendix}
\usepackage{amssymb}
\usepackage[all]{xy}
\usepackage{amscd}
\usepackage{color}
\usepackage{extarrows}
\usepackage{multirow}
\usepackage{enumitem}
\usepackage[
        hypertexnames=false,
        hyperindex,
        bookmarks=false,
        colorlinks=true,
        linkcolor=magenta,
        citecolor=magenta,
        urlcolor=blue
]{hyperref}
\usepackage{bm}

\newtheorem{theorem}{Theorem}[subsection]
\newtheorem{lemma}[theorem]{Lemma}
\newtheorem{proposition}[theorem]{Proposition}
\newtheorem{corollary}[theorem]{Corollary}
\theoremstyle{definition}
\newtheorem{definition}[theorem]{Definition}
\newtheorem{example}[theorem]{Example}
\newtheorem{remark}[theorem]{Remark}

\newtheorem{conjecture}[theorem]{Conjecture}

\theoremstyle{plain}
\newtheorem{theor}{Theorem}[section]
\theoremstyle{plain}

\theoremstyle{plain}

\theoremstyle{plain}
\newtheorem{conje}[theor]{Conjecture}

\theoremstyle{definition}
\newtheorem{rema}{Remark}
\numberwithin{equation}{section}

\makeatletter
\newcommand*{\xnRightarrow}[2][]{%
  \ext@arrow 0359\nRightarrowfill@{#1}{#2}%
}
\newcommand*{\nRightarrowfill@}{%
  \narrowfill@\Relbar\Relbar\Rightarrow\neq
}
\newcommand*{\narrowfill@}[5]{%
  $\m@th\thickmuskip0mu\medmuskip\thickmuskip\thinmuskip\thickmuskip
  \relax#5#1\mkern-7mu%
  \cleaders\hbox{$#5\mkern-2mu#2\mkern-2mu$}\hfill
  \mkern-5mu %
  #4%
  \mkern-5mu %
  \cleaders\hbox{$#5\mkern-2mu#2\mkern-2mu$}\hfill
  \mkern-7mu#3$%
}
\makeatother

\title{Generalization of semi-regular sequences:\\
Maximal Gr\"{o}bner basis degree, variants of genericness, and related conjectures\thanks{This is a revised and extended version of \cite{KY2}.}}
\date{\today}
\author{Momonari Kudo\thanks{Fukuoka Institute of Technology, 3-30-1 Wajiro-higashi, Higashi-ku, Fukuoka, 811-0295 Japan, \texttt{m-kudo@fit.ac.jp}} \and Kazuhiro Yokoyama\thanks{Rikkyo University, 3-34-1 Nishi-Ikebukuro, Toshima-ku, Tokyo, 171-8501 Japan, \texttt{kazuhiro@rikkyo.ac.jp}}}
\begin{document}

\maketitle



\abstract{
Nowadays, the notion of semi-regular sequences, originally proposed by Fr\"oberg, becomes 
very important not only in Mathematics, but 
also in Information Science, \textcolor{black}{in particular} Cryptology. 
For example, it is highly expected that 
randomly generated 
\textcolor{black}{polynomials} form a semi-regular sequence, and 
based on this observation, secure cryptosystems 
based on polynomial systems \textcolor{black}{can be} \textcolor{black}{devised}. 
In this paper, we deal with a semi-regular sequence
and its extension, named a generalized 
cryptographic semi-regular sequence, and give precise analysis on 
the complexity of computing a Gr\"obner basis 
of the ideal generated by such a sequence 
with help of several regularities of the ideal 
related to Lazard's bound on maximal Gr\"{o}bner basis degree and other bounds. 
We also study the genericness of 
the property that a sequence is semi-regular, 
and its variants related to Fr\"oberg's 
conjecture. Moreover, we discuss on 
the genericness of another important property 
that 
the initial ideal is weakly reverse lexicographic, 
related to Moreno-Soc\'{i}as' conjecture, 
and show some criteria to examine whether both 
Fr\"oberg's conjecture and Moreno-Soc\'{i}as' 
one hold at the same time. 
}

\section{Introduction}\label{sec:Intro}

In the theory of polynomial rings over fields, Gr\"{o}bner bases introduced first by Buchberger~\cite{Buchberger} provide various methods for investigating important properties of an ideal and its associated algebraic variety.
Here, a Gr\"{o}bner basis is defined as a special kind of its generator set, relative to a monomial ordering.
Buchberger also presented the first algorithm (Buchberger's algorithm) for computing Gr\"{o}bner bases.
Lazard's formulation in \cite{Lazard} and \cite{Lazard81} relates the problem of computing Gr\"{o}bner bases to the problem of reducing so-called Macaulay matrices.
FGLM~\cite{FGLM} and Gr\"{o}bner walk~\cite{GroebnerWalk} were {proposed} as algorithms that efficiently convert a given Gr\"{o}bner basis to a Gr\"{o}bner basis with respect to another monomial ordering; Hilbert driven {algorithm}~\cite{Tra} can be also applied to such a conversion.
{Faug\`{e}re} proposed his innovative algorithms $F_4$~\cite{F4} and $F_5$~\cite{F5} allowing one to quite efficiently compute Gr\"{o}bner bases, and these are nowadays state-of-the-art algorithms.
We refer to \cite{EF} for a survey on signature-based algorithms including $F_5$.
Thanks to Gr\"{o}bner basis conversion algorithms, it is important to estimate the complexity of computing a Gr\"{o}bner basis with respect to a specific monomial ordering.
\textcolor{black}{(Throughout this paper, all the complexities are measured by the number of arithmetic operations in a field to which all the coefficients of the input polynomials belong.)}
In this paper, we consider only the case of {\it graded} monomial orderings; \textcolor{black}{in particular, we adopt a graded reverse lexicographic ordering.}

In order to estimate the complexity, it suffices to provide an upper-bound on the so-called {\it solving degree}:
Roughly speaking, the solving degree is the highest degree of polynomials involved in Gr\"{o}bner basis computation, see \cite[Section 2.2]{KY24b} for rigorous definitions.
If the {given} ideal is homogeneous, the solving degree can be equal to the maximal degree of the reduced Gr\"{o}bner basis with respect to a monomial ordering that one considers, and its upper-bounds have been well-studied.
In the case of inhomogeneous ideals, the solving degree can be {\it greater than} the maximal degree of the reduced Gr\"{o}bner basis, whence estimating it {\it directly} is difficult in general, see \cite{CG23} and \cite{KY24b} for details.
On the other hand, it is well-known (cf.\ \cite[Prop.\ 4.3.18]{KR}) that a Gr\"{o}bner basis of an inhomogeneous ideal with respect to a graded monomial ordering is easily obtained from a Gr\"{o}bner basis of the ideal of {\it homogenized} generators.
Therefore, the homogeneous case is {essential}, and we focus on that case.
\textcolor{black}{
In the following, let $K$ be a field, and let $R=K[x_1,\ldots,x_n]$ be the polynomial ring of $n$ variables over $K$.
The term ``degree'' means total degree, unless otherwise noted.
Let $R_d$ be its homogeneous part of degree $d$ of $R$ (i.e., the $K$-linear space spanned by the monomials in $R$ of degree $d$) for each non-negative integer $d$.
For an ideal (possibly inhomogeneous) $I$ of $R$}, we denote by $\mathrm{max.GB.deg}_{\prec}(I)$ the maximal degree of the elements of the reduced Gr\"{o}bner basis of $I$ with respect to a monomial ordering $\prec$ on the set {$\mathcal{T}$} 
of monomials in $R$, following \cite[Definition 7]{CG20} (it is also denoted by $\deg (I,\prec)$ in the literature of commutative algebra, see e.g., \cite[Definition 4]{HS}).
If $I$ is generated by $F \subset R$, we may write $\mathrm{max.GB.deg}_{\prec}(F)$ in place of it.

\subsubsection*{Lazard's bound and Hashemi-Seiler's inequality}

For a {\it homogeneous} ideal $I \subset R=K[x_1,\ldots,x_n]$ with a generating set $F=\{f_1,\ldots,f_m\} \subset R$, it is well-known {(see e.g., \cite[Proposition 1]{BFS-F5}, \cite[Equation (24)]{Steiner24})} that a Gr\"{o}bner basis of $I$ with respect to a graded monomial ordering up to degree $d$ can be computed in $O(m d \binom{n+d-1}{d}^{\omega})$, where $2\leq \omega < 3$ is the exponent of matrix multiplication.
(We can also obtain another upper-bound $O(m \binom{n+d}{d}^{\omega})$ by the proof of \cite[Theorem 1.72]{Spa}.)
{From} this fact, {it follows} in the {\it homogeneous} case that the complexity of computing the reduced Gr\"{o}bner basis of $I$ is upper-bounded by $O(m D \binom{n+D-1}{D}^{\omega})$ (or $O(m \binom{n+D}{D}^{\omega})$) for any integer $D$ with $\mathrm{max.GB.deg}_{\prec}(I) \leq D$.
Hence, finding a tighter bound $D$ of $\mathrm{max.GB.deg}_{\prec}(I)$ is important for estimating {shapely} the complexity of Gr\"{o}bner basis computation in the homogeneous case.
The case on which we {focus} is the case where the Krull dimension {of $R/I$, denoted by $\mathrm{Krull.dim} (R/I)$,} is equal to or less than one; see e.g., \cite[Section 1]{HS} for a summary of existing results when {$\mathrm{Krull.dim} (R/I) \geq 2$}.
In the case where $\mathrm{Krull.dim} (R/I) \leq 1$, Lazard's bound (called also Macaulay's bound) is well-known.
{Specifically, if $K$ is large enough, then there exists a (suitable) linear coordinate change $\sigma$ such that $R/(I^{\sigma}+\langle x_n \rangle)$ is Artinian (equivalently $I$ is in {\it quasi stable position}, see Lemma \ref{lem:KrullDim1Equi} below).
{From this, we assume without loss of generality that $R/(I+\langle x_n \rangle)$ is Artinian.}
In this case, putting $d_j = \deg(f_j)$ for $1\leq j \leq m$ and supposing that $d_1 \geq d_2 \geq \cdots \geq d_m$ (descending order), it follows that}
\begin{equation}\label{eq:Lazard}
    \mathrm{max.GB.deg}_{\prec}(I) \leq \sum_{j=1}^{\min\{m,n\}}(d_j-1)+1
\end{equation}
{for the graded reverse lexicographical ordering $\prec$ with $x_n \prec \cdots \prec x_1$, see \cite[Th\'eor\`eme 3.3]{Lazard81} and \cite[Theorem 2]{Lazard}.}
{Recently, Hashemi-Seiler~\cite{HS} re-proved Lazard's bound by using the inequality}
\begin{equation}\label{eq:HS}
    \mathrm{max.GB.deg}_{\prec}(I) \leq \mathrm{max} \{\mathrm{hilb}(I+\langle x_n\rangle),\mathrm{hilb}(I) \},
\end{equation}
where $\mathrm{hilb}(J)$ denotes the Hilbert regularity of a homogeneous ideal $J$ of $R$.
Note that $\mathrm{hilb}(I+\langle x_n \rangle)$ is equal to the {\it degree of regularity} ${d}_{\rm reg}(I+\langle x_n \rangle)$ since $R/(I+\langle x_n \rangle)$ is Artinian, and $\mathrm{hilb}(I)$ is equal to the {\it generalized degree of regularity} $\widetilde{d}_{\rm reg}(I)$.
For the definition of $d_{\rm reg}$ and {$\widetilde{d}_{\rm reg}$}, see Remark \ref{rem:relation} below. 

\begin{rema}
    \textcolor{black}{
    Note that the condition $\mathrm{Krull.dim} (R/I) \leq 1$ \textcolor{black}{does not necessarily imply} $d_{\rm reg}(I+\langle x_n \rangle) = \mathrm{hilb}(I+\langle x_n \rangle)<\infty$, without considering any linear coordinate change $\sigma$.
    For example, putting $F:=\{f_1:=x_1x_2-1,f_2:=x_2-1\}\subset K[x_1,x_2]$ (inhomogeneous), it is obvious that $\mathrm{Krull.dim} (R[y]/I) = 1$ for $I=\langle F^h \rangle$ but $d_{\rm reg}(I+\langle y \rangle) =d_{\rm reg}(\langle F^{\rm top}\rangle)=\infty$, where $F^h$ is the homogenization of $F$ by an extra variable $y$, and where $F^{\rm top}$ is the homogeneous part of highest degree, see Remark \ref{rem:inhomo} below.}
\textcolor{black}{In this case, $F^h=\{x_1x_2-y^2,x_2-y\}$ and $F^{\rm top}=\{x_1x_2,x_2\}$. }
    \textcolor{black}{
    On the other hand, if $K$ is a finite field of order $q$ and if we would compute a zero of $\langle F\rangle$ over $K$, we often add the set of so-called field equations $x_i^q-x_i$ to an inhomogeneous $F$, say $F_1 := F \cup \{ x_i^q-x_i : 1 \leq i \leq n\}$, so that we always have $d_{\rm reg}(\langle F_1^h\rangle+\langle y \rangle) =d_{\rm reg}(\langle F_1^{\rm top}\rangle)<\infty$.}
\end{rema}

\subsubsection*{Our results and conjectures}
{In this paper, we focus on several cases that are expected to be {\it generic} in some sense, and {obtain upper-bounds on $\mathrm{max.GB.deg}_{\prec}(I)$ and some results related to Fr\"{o}berg's conjecture~\cite{Froberg} and Moreno-Soc\'{i}as' one~\cite{MS}, in those cases.}
For example, we will assume that $R/(I+\langle x_n \rangle)$ has the Hilbert series associated to generic forms, as conjectured by Fr\"{o}berg~\cite{Froberg} {(see Conjecture \ref{conj:Froberg} below)}.
Note that a homogeneous polynomial sequence is called {\it cryptographic semi-regular} (cf.\ \cite{BNDGMT21}) if its associated coordinate ring has the Hilbert series conjectured by Fr\"{o}berg, see Definition \ref{def:csemireg} and Proposition \ref{prop:Diem} below for details.
{In the literature, a cryptographic semi-regular sequence plays a very important role in analyzing Gr\"{o}bner basis computation, both in theory and in practical applications to cryptography, see e.g., Bardet et al's works~\cite{Bar}, \cite{BFS}, \cite{BFS-F5}, \cite{BFSY}, Gorla et al.'s works~\cite{BNDGMT21}, \cite{CG20}, \cite{CG23}, and our ones~\cite{KY}, \cite{KY24b}.}
In our earlier work~\cite{KY24b}, when $m \geq n$, we improved Lazard's bound on $\mathrm{max.GB.deg}_{\prec}(F)$ to $\sum_{j=1}^{n}(d_j-1)+1$ with $d_1 \leq \cdots \leq d_{m}$ {(ascending order)}, supposing that {$(f_1|_{x_n=0},\ldots,f_i|_{x_n=0})$} is cryptographic semi-regular \textcolor{black}{on $K[x_1,\ldots,x_{n-1}]$} for each $1 \leq i \leq m$.}
{In \cite{KY24b}, we also defined a {\it generalized cryptographic semi-regular} sequence (see Definition \ref{def:gensemireg} below for the definition and Remark \ref{rem:construction} for a construction), as an extension of cryptographic semi-regular sequence to the case of Krull dimension one.}

{As the first result of this paper, we prove Theorem \ref{thm:mainT}, which provides an upper-bound on $\mathrm{max.GB.deg}_{\prec}(I)$ for a generalized cryptographic semi-regular sequence and a complexity bound on computing a Gr\"{o}bner basis of $I$.}

\begin{theor}\label{thm:mainT}
    Let $R=K[x_1,\ldots,x_n]$ be the polynomial ring of $n$ variables over a field $K$ (in any characteristic), and $\prec$ the graded reverse lexicographical ordering with $x_n \prec \cdots \prec x_1$.
    Let $I$ be a homogeneous ideal of $R$ generated by homogeneous polynomials $f_1,\ldots,f_m$ in $R\smallsetminus K$ such that $\mathrm{Krull.dim}(R/I)\leq 1$.
    Assume that {$(f_1|_{x_n=0},\ldots,f_m|_{x_n=0})$ is cryptographic semi-regular \textcolor{black}{on $K[x_1,\ldots,x_{n-1}]$} (equivalently $(f_1,\ldots,f_m,x_n)$ is cryptographic semi-regular)}, and that $(f_1,\ldots,f_m)$ is generalized cryptographic semi-regular, 
    Then the following (1) and (2) hold:
    \begin{enumerate}
        \item[(1)] We have
        \begin{equation}\label{eq:max_Dnew}
        \mathrm{max.GB.deg}_{\prec}(I) \leq \mathrm{max}\{ d_{\rm reg}(I+\langle x_n \rangle), \widetilde{d}_{\rm reg}(I) \} \leq D^{(n,m)},
        \end{equation}
        where 
        \begin{equation}\label{eq:newbound}
        D^{(n,m)} := \left\{ \! \begin{array}{cc}\displaystyle {\deg\left(\left[\frac{\prod_{i=1}^m(1-z^{d_i})}{(1-z)^{n}}\right]\right)}+1 & (m\geq n), \\
        \displaystyle \sum_{i=1}^{n-1} (d_i-1) + 1 & (m=n-1).
        \end{array}\right.
        \end{equation}
        \item[(2)] The complexity of computing a Gr\"{o}bner basis of $I$ with respect to $\prec$ is upper-bounded by
        \begin{equation}\label{eq:new_comp_bound}
        O \left( m \binom{n+D^{(n,m)}-1}{D^{(n,m)}}^{\omega} \right)
        \end{equation}
        \textcolor{black}{arithmetic operations in $K$}, where $2\leq \omega < 3$ is the exponent of matrix multiplication.
    \end{enumerate}
\end{theor}

The bound $D^{(n,m)}$ in \eqref{eq:max_Dnew} is optimal for a homogeneous ideal $I$ satisfying the assumptions of Theorem \ref{thm:mainT}, see Example \ref{ex:optimal} below.
\textcolor{black}{The proof of (2) in Theorem \ref{thm:mainT} is based on a well-known fact that the a Gr\"{o}bner basis of $I$ is obtained by computing the reduced row echelon form of the degree-$D$ Macaulay matrix for any $D$ with $D \geq  \mathrm{max.GB.deg}_{\prec}(I)$, see Section \ref{subsec:general} for details.
}

Our assumption that a homogeneous polynomial sequence is generalized cryptographic semi-regular could be useful to analyze the security of multivariate cryptosystems and so on (cf.\ \cite[Section 4.3]{KY24b}).
The bound {$D^{(n,m)}$} given in \eqref{eq:newbound} coincides with Lazard's bound when {$m\in \{n-1, n\}$}, while it can be sharper than Lazard's one when $m>n$.
We also note that, the complexity bound \eqref{eq:new_comp_bound} is {slightly} better than the usual bound $O(mD\binom{n+D-1}{D}^{\omega})$ (see e.g., \cite[Proposition 1]{BFS-F5}, \cite[Equation (24)]{Steiner24}) and than $O(m\binom{n+D}{D}^{\omega})$ (which comes from the proof of \cite[Theorem 1.72]{Spa}) for an arbitrary $D$ with $D \geq \mathrm{max.GB.deg}_{\prec}(F)$.

As the second result, we shall prove that Hashemi-Seiler's inequality \eqref{eq:HS} is as possible as tight under some assumptions, {one of which is that an initial ideal is \textit{weakly reverse lexicographic}.
Here, a weakly reverse lexicographic ideal is a monomial ideal $J$ such that if $x^{\alpha}$ is one of the minimal generators of $J$ then every monomial of the same degree which preceeds $x^{\alpha}$ must belong to $J$ as well.
Moreno-Soc\'{i}as~\cite{MS} conjectured that, for the graded reverse lexicographical ordering, the initial ideal of a generic ideal is a weakly reverse lexicographic.}

\begin{theor}\label{thm:WRL-GB}
    Let $R=K[x_1,\ldots,x_n]$ be the polynomial ring of $n$ variables over a field $K$ (in any characteristic), and let $\prec$ be the graded reverse lexicographical ordering $\prec$ on $R$ with $x_n \prec \cdots \prec x_1$.
    Let $I$ be a proper homogeneous ideal of $R$ {with $\mathrm{Krull.dim} (R/I) \leq 1$ and $\mathrm{Krull.dim}(R/(I+\langle x_n \rangle))=0$.}
    {If {$(f_1|_{x_n=0},\ldots,f_m|_{x_n=0})$ is cryptographic semi-regular \textcolor{black}{on $K[x_1,\ldots,x_{n-1}]$} (equivalently $(f_1,\ldots,f_m,x_n)$ is cryptographic semi-regular)} and if} 
    {the initial ideal ${\rm in}_{\prec}(I)=\langle \mathrm{LM}_{\prec}(I) \rangle_R$} 
    is weakly reverse lexicographic, then the inequality \eqref{eq:HS} becomes an equality.
\end{theor}

{
After proving Theorems \ref{thm:mainT} and \ref{thm:WRL-GB}, in Section \ref{sec:generic}, we also study the genericness of our assumptions used in the theorems, based on the theory of {\it parametric Gr\"{o}bner bases}.
For this, in \textcolor{black}{Section} \ref{subsec:generic}, we formulate the genericness of cryptographic semi-regular sequences in terms of parametric Gr\"{o}bner bases.
Then, in Section \ref{sec:generic}, some analogous results for generalized cryptographic semi-regular sequences will be proved, and the following conjectures will be raised as variants of Fr\"{o}berg's conjecture:
}

\begin{conje}[{Conjecture \ref{conj:A}}; see also {\cite[Conjecture 4.3.4]{KY24b}}]\label{conj:1}
    {Assume that $K$ is an infinite field, and we endow $V_{n,m,(d_1,\ldots,d_m)}:=R_{d_1}\times \cdots \times R_{d_m}$ naturally with the Zariski topology.
    For {any} $n$, $m$, and $(d_1,\ldots,d_m)$ and for any $\bm{p}\in \mathbb{P}^{n-1}(K)$, the property that $\bm{F}\in V_{n,m,(d_1,\ldots,d_m)}$ is generalized cryptographic semi-regular is generic on $Z_{\bm{p}}:=\{ \bm{H} \in V_{n,m,(d_1,\ldots,d_m)} : \bm{p}\in V_{\overline{K}}(\bm{H}) \}$, namely there exists a non-empty Zariski-open set of $Z_{\bm{p}}$ with respect to the induced topology on which the property holds.}
\end{conje}


\begin{conje}[{Conjecture \ref{conj:B}}]\label{conj:2}
    {Assume that $K$ is an infinite field.
    For {any} $n$, $m$, and $(d_1,\ldots,d_m)$, the property that $\bm{F}\in V_{n,m,(d_1,\ldots,d_m)}$ is generalized cryptographic semi-regular is generic on $V_{n,m,(d_1,\ldots,d_m)}^{\dim \geq 1}:=\{ \bm{H} \in V_{n,m,(d_1,\ldots,d_m)} : \mathrm{Krull.dim}(R/\langle \bm{H}\rangle)\geq 1 \}$, namely there exists a non-empty Zariski-open set of $V_{n,m,(d_1,\ldots,d_m)}^{\dim \geq 1}$ with respect to the induced topology on which the property holds.}
\end{conje}

{We also show a new criterion to examine whether Fr\"{o}berg's conjecture and Moreno-Soc\'{i}as' one hold at the same time, \textcolor{black}{for a \textit{fixed} parameter set $(n,m,d_1,\ldots,d_m)$ (although these conjectures are set for {\it any} parameter set, see Remark \ref{rem:any} for details):}}

\begin{theor}[Theorem \ref{thm:semi+wrl}]\label{th:semi+wrl}
    {Assume that $K$ is an infinite field, and let $n,m,d_1,\ldots,d_m$ be positive integers.
    If there exists a cryptographic semi-regular sequence $\bm{F}\in V_{n,m,(d_1,\ldots,d_m)}$ such that 
   {${\rm in}_{\prec}(\langle \bm{F}\rangle)$} is weakly reverse lexicographic, then there also exists a non-empty Zariski-open set $U$ of $V_{n,m,(d_1,\ldots,d_m)}$ such that $\bm{H}$ is cryptographic semi-regular and 
   ${\rm in}_{\prec}(\langle \bm{H}\rangle)$ is weakly reverse lexicographic for any $\bm{H}\in U$.}
   {Namely, Fr\"{o}berg's conjecture and Moreno-Soc\'{i}as' one both hold for the parameter \textcolor{black}{set} $(n,m,d_1,\ldots,d_m)$.}
\end{theor}

\begin{rema}\label{rem:any}
    \textcolor{black}{Pardue proved in his seminal paper~\cite{Pardue} that Fr\"{o}berg's conjecture is equivalent to his several conjectures, and that Moreno-Soc\'{i}as' conjecture implies all of them.
    Note that these equivalence and implication are intended to mean that each conjecture holds for {\it any} choice of parameters $(n,m,d_1,\ldots,d_m)$.
    More precisely, we see from the proof of his theorem~\cite[Theorem 2]{Pardue} that (1) and (2) are equivalent to each other and that (3) implies (1) and (2), over an infinite field:
    \begin{enumerate}
        \item[(1)] (Fr\"{o}berg's conjecture) The property that $\bm{F}\in V_{n,m,(d_1,\ldots,d_m)}$ is cryptographic semi-regular is generic (in terms of Definition \ref{def:generic}) for any $(n,m,d_1,\ldots,d_m)$.
        \item[(2)] (Pardue's conjecture~\cite[Conjecture B]{Pardue}) The property that $\bm{F}\in V_{n,m,(d_1,\ldots,d_m)}$ is semi-regular is generic for any $(n,m,d_1,\ldots,d_m)$.
        \item[(3)] (Moreno-Soc\'{i}as' conjecture~\cite{MS}) The property that the initial ideal of the homogeneous ideal generated by $\bm{F}\in V_{n,m,(d_1,\ldots,d_m)}$ with respect to a graded reverse lexicographical ordering is weakly reverse lexicographic is generic for any $(n,m,d_1,\ldots,d_m)$.
    \end{enumerate}
    Namely, \textcolor{black}{even} if the property in (3) is \textcolor{black}{shown to be} generic for a fixed parameter set $(n,m,d_1,\ldots,d_m)$, it is unclear that the property in (1) (or (2)) is generic for the same parameter set (this also applies to the case of (1) $\Rightarrow$ (2)).}
\end{rema}

{Analogously to Theorem \ref{th:semi+wrl}, we finally prove the following theorem, which can also examine Conjecture \ref{conj:1}:}

\begin{theor}[Proposition \ref{prop:gsemi+wrl} and Proposition \ref{prop:Z_p}]\label{th:gsemi+wrl}{
Under the same setting as in Theorem \ref{th:semi+wrl}, assume that there exists a generalized cryptographic semi-regular sequence $\bm{F}$ in $V_{n,m,(d_1,\ldots,d_m)}$ with $\mathrm{Krull.dim}(R/\langle \bm{F}\rangle)=1$ such that $\dim_K (R/\langle \bm{F}\rangle)_D=1$ for $D=\widetilde{d}_{\rm reg}(\langle \bm{F}\rangle)$ and 
${\rm in}(\langle \bm{F}\rangle)$ is weakly reverse lexicographic.
Then, for any $\bm{p}\in \mathbb{P}^{n-1}(K)$, there exists a non-empty Zariski-open set $U_{\bm{p}}$ of $Z_{\bm{p}}:=\{ \bm{H} \in V_{n,m,(d_1,\ldots,d_m)} : \bm{p}\in V_{\overline{K}}(\bm{H}) \}$ with respect to the induced topology such that $\bm{H}$ is generalized cryptographic semi-regular and ${\rm in}(\langle \bm{H}\rangle)$ is weakly reverse lexicographic for any $\bm{H}$ in $U_{\bm{p}}$.}
\end{theor}


{To the end of Introduction, we remark an application of Theorem \ref{thm:mainT} to the inhomogeneous case.}

\begin{rema}[Inhomogeneous case]\label{rem:inhomo}
    By Theorem \ref{thm:mainT}, we can also estimate the complexity of computing a Gr\"{o}bner basis for an inhomogeneous polynomial sequence defining a zero-dimensional ideal, as in \cite[Section 4.3]{KY24b}.
    Indeed, for a subset $F=\{f_1,\ldots,f_m \}$ of $R=K[x_1,\ldots,x_n]$ containing at least one inhomogeneous polynomial, a Gr\"{o}bner basis of $\langle F \rangle_R$ is obtained by reducing the associated degree-$D$ Macaulay matrix at $D = \mathrm{max.GB.deg}_{\prec^h}(I)$ with $I=\langle f_1^h,\ldots,f_m^h \rangle \subset R':=R[y]$, see \cite[Theorem 7]{CG20}.
    Here, each $f_j^h$ denotes the homogenization of $f_j$ by an extra variable $y$, and $\prec^h$ is the homogenization of $\prec$, see e.g., \cite[Section 3.1.1]{KR} for definition. 
    The complexity of reducing the Macaulay matrix can be easily estimated as $O(m\binom{n+D}{D}^{\omega})$, without assuming $D=O(n)$.
    Suppose the following three conditions:
    \begin{itemize}
        \item The Krull dimension of $R'/\langle f_1^h,\ldots,f_m^h,y \rangle$ is zero, namely the Krull dimension of $R/\langle f_1^{\rm top},\ldots,f_m^{\rm top} \rangle$ is zero, where $f_j^{\rm top}:=(f_j^h)|_{y=0}$ for $1 \leq j \leq m$.
        (Equivalently these two graded rings are both Artinian.)
        \item The sequence $(f_1^h,\ldots,f_m^h,y) \in (R')^{m+1}$ is cryptographic semi-regular, equivalently $(f_1^{\rm top},\ldots,f_m^{\rm top}) \in R^m$ is cryptographic semi-regular.
        \item The sequence {$(f_1^h,\ldots,f_m^h) \in (R')^{m}$} is generalized cryptographic semi-regular.
    \end{itemize}
    {Then, it follows from Theorem \ref{thm:mainT} that $\max\{ d_{\rm reg}(\langle F^{\rm top}\rangle),\widetilde{d}_{\rm reg}(\langle F^h\rangle)\}\leq D^{(n+1,m)}$ with $F^{\rm top}:=\{f_1^{\rm top},\ldots,f_m^{\rm top}\}$ and we have the complexity bound
    \begin{equation}\label{eq:new_comp_bound_inhom}
        O \left( m \binom{n+D^{(n+1,m)}}{D^{(n+1,m)}}^{\omega} \right).
    \end{equation}
    Note that this upper-bound is heuristically found in \cite{FK24}.}
\end{rema}

\numberwithin{equation}{subsection}

\section{Preliminaries}

Let $K$ be a field, and let $\overline{K}$ denote its algebraic closure.
Let $n$ be a natural number, and let $R =K[x_1,\ldots,x_n]$ be the polynomial ring of $n$ variables over $K$.
\textcolor{black}{{We denote by $\mathcal{T}$ the set of monomials in $\{x_1,\ldots,x_n\}$. 
As our convention, as graded structure objects, 
for a subset $S$ of $R$ and a non-negative integer $d$, 
we also denote by $S_d$ the subset consisting of all elements of degree $d$ in $S$.
As an exception, if $S$ is an ideal of $R$, it is allowed that $S_d$ contains $0$.}}
The Krull dimension of a commutative ring $A$ is denoted by $\mathrm{Krull.dim}(A)$.
For ideals $I$ and $J$ of $R$, their ideal quotient is defined as $(I : J) = \{ f \in R \mid f J \subset I \}$.
The saturation of $I$ with respect to $J$ is $(I : J^{\infty}) := \{ f \in R \mid f J^s \subset I \mbox{ for some $s \geq 1$}\}$.
In particular, the saturation of a homogeneous ideal $I$ with respect to the maximal homogeneous ideal $\mathfrak{m}=\langle x_1,\ldots,x_n \rangle_R$ is denoted by $I^{\rm sat} = (I:\mathfrak{m}^{\infty})$, which is also homogeneous.

\subsection{Hilbert regularity}

Let $I$ be a non-zero homogeneous ideal of $R$.
Let ${\rm HP}_{R/I} \in \mathbb{Q}[z]$ denote the Hilbert polynomial of $R/I$, namely there exists a non-negative integer $d_0$ such that ${\rm HF}_{R/I}(d) = {\rm HP}_{R/I}(d)$ for any $d$ with $d \geq d_0$, where $\mathrm{HF}$ is the Hilbert function.

\begin{definition}[Hilbert regularity (or index of regularity~{\cite[Chapter IX, Section 3]{CLO}})]
    The smallest $d_0$ such that ${\rm HF}_{R/I}(d) = {\rm HP}_{R/I}(d)$ for any $d$ with $d \geq d_0$ is called the {\it Hilbert regularity} of $I$, denoted by ${\rm hilb}(I)$, or is called the {\it index of regularity} of $I$, denoted by $i_{\rm reg}(I)$.
\end{definition}


Throughout the rest of this subsection, we set $r:=\mathrm{Krull.dim}(R/I)$.

\begin{remark}\label{rem:relation}
        It is well-known that the following conditions are all equivalent:
        \begin{enumerate}
            \item[(1)] The Krull dimension $r$ of $R/I$ is equal to $1$ (resp.\ $0$).
            \item[(2)] The projective algebraic set $V_{\overline{K}}(I)$ is a finite set (resp.\ an empty set).
            \item[(3)] The Hilbert polynomial ${\rm HP}_{R/I}$ is a positive constant (resp.\ the zero).
            \item[(4)] There exists a non-negative integer $d_0$ such that \textcolor{black}{${\rm HF}_{R/I}(d) = {\rm HF}_{R/I}(d_0)>0$} (resp.\ ${\rm HF}_{R/I}(d) = 0$) for any $d$ with $d \geq d_0$.
        \end{enumerate}
        In this case, it is straightforward that
        \begin{eqnarray*}
            \mathrm{hilb}(I) &=&\min \{ d_0 : {\rm HF}_{R/I}(d) = {\rm HF}_{R/I}(d_0) \mbox{ for any $d$ with $d \geq d_0$}\}\\
            &=& \min \{ d_0 : \dim_K (R/I)_d = \dim_K (R/I)_{d_0} \mbox{ for any $d$ with $d \geq d_0$}\}.
        \end{eqnarray*}
        The middle and the right hand sides are called the {\it generalized degree of regularity} of $I$ for $r \leq 1$, denoted by $\widetilde{d}_{\rm reg}(I)$ (cf.\ \cite[Sect.\ 2.3]{KY24b}):
        When $r=0$, these are equal to the {\it degree of regularity} $d_{\rm reg}(I) := \min \{ d : R_d = I_d \}$.
        Namely, we have
        \begin{equation}\label{eq:Hilb}
        \mathrm{hilb}(I) = \left\{ 
        \begin{array}{cll}
             \widetilde{d}_{\rm reg}(I) & &  (r = 1), \\
             \widetilde{d}_{\rm reg}(I) & = d_{\rm reg}(I) & (r=0).
        \end{array}\right.
        \end{equation}
        See \cite{BNDGMT21} for bounds on ${d}_{\rm reg}(I)$.
\end{remark}

Although we mainly treat the case where $r \leq 1$ in this paper, we distinguish $\mathrm{hil}(I)$ and $\widetilde{d}_{\rm reg}(I)$ for the compatibility with the condition that a sequence of homogeneous polynomials in $R$ is regular (cf.\ Lemma \ref{lem:four} below).


\begin{remark}
    {Note also} that \textcolor{black}{the Hilbert series} ${\rm HS}_{R/I}$ is written as {${\rm HS}_{R/I}(z) =h_{R/I}(z)/{(1-z)^{r}}$} for some polynomial $h_{R/I}(z) \in \mathbb{Z}[z]$, which is called the {\it $h$-polynomial} of $R/I$.
    Then, we have ${\rm hilb}(I) = \deg h_{R/I}-r+1$, see \cite[Prop.\ 4.1.12]{BH}.

    Moreover, it follows from Grothendieck-Serre's formula that
    \[
    {\rm hilb}(I) \leq \mathrm{reg}(I) = \mathrm{reg}(R/I) + 1,
    \]
    where `${\rm reg}$' means {\it Castelnuovo-Mumford regularity}, see \cite[\S20.5]{Eisen} for the definition.
    The equality holds if $R/I$ is Artinian, i.e., the Krull dimension of $R/I$ is zero{, see e.g., \cite[Corollary 4.4]{Eisen2}}.
\end{remark}

\begin{definition}[cf.\ {\cite[Section 3]{HS}}]
    We say that $I$ is in {\it Noether position} with respect to the variables $x_1,\ldots,x_{n-r}$, where $r=\mathrm{Krull.dim}(R/I)$, if the injective ring extension $K[x_{n-r+1}, \ldots , x_{n}] \to R/I$ is integral.
\end{definition}

\begin{lemma}[{\cite[Lemma 4.1]{BG01}}; see also {\cite[Proposition 5.4]{LJ}}]\label{lem:NPequi}
Let $I$ be a homogeneous ideal of $R$, and let $\prec$ be the {graded reverse lexicographical} {ordering} on $R$ with $x_n \prec \cdots \prec x_1$.
Then, the following are equivalent:
\begin{enumerate}
    \item[(1)] $I$ is in Noether position with respect to the variables $x_1,\ldots,x_{n-r}$.
    \item[(2)] For each $i$ with $1\leq i \leq n-r$, there exists $r_i \geq 0$ such that $x_i^{r_i} \in \langle \mathrm{LM}_{\prec}(I) \rangle$.
    \item[(3)] \textcolor{black}{$\mathrm{Krull.dim} (R/ (I+\langle x_{n-r+1},\ldots,x_n\rangle)) = 0$}.
    \item[(4)] \textcolor{black}{$\mathrm{Krull.dim} (R/ (\langle \mathrm{LM}_{\prec}(I) \rangle+\langle x_{n-r+1},\ldots,x_n\rangle)) = 0$}.
\end{enumerate}
\end{lemma}

In the proof of \cite[Lemma 4.1]{BG01}, the authors first proved the equivalence of the conditions (2), (3), and (4), and the implication (1) $\Rightarrow$ (2).
The non-trivial part is that (2), (3), and (4) imply (1), which follows from \cite[Proposition 5.4]{LJ}.

Here, we also recall the definition introduced by Hashemi~\cite{Hashemi2010} that a homogeneous ideal is in strong Noether position.

\begin{definition}[{\cite[{Definition 2.4 and} Theorem 2.19]{Hashemi2010}}]
    A homogeneous ideal $I$ of $R$ is said to be in {\it strong Noether position} (SNP) with respect to the ordered variables $x_1,\ldots,x_n$ if we have $I^{\rm sat} = (I:x_n^{\infty})$ and
    \[
    {(I + \langle x_{n-i+1},\ldots,x_n \rangle)^{\rm sat}} = ( {(I+\langle x_{n-i+1},\ldots,x_n \rangle)} : x_{n-i}^{\infty})
    \]
    for any $i$ with \textcolor{black}{$1 \leq i \leq r-1$}.
\end{definition}

For equivalent definitions, see \cite[Section 2]{Hashemi2010} for details.
Note that $I$ is in SNP if and only if $\langle \mathrm{LM}_{\prec}(I) \rangle$ is in SNP {(cf.\ \cite[Proposition 2.11]{Hashemi2010})}.
Hashemi also proved the following inequality:

\begin{theorem}[{\cite[Proposition 4.7 and Theorem 4.17]{Hashemi2010}}]\label{thm:SNPHilb}
    If $I$ is in SNP, then $\mathrm{reg}(I)$ is equal to
    \[
    \overline{\rm hilb}(I) := \max \{ \mathrm{hilb}(I) \} \cup \{ \mathrm{hilb}(\langle I, x_{n-i+1},\ldots,x_n \rangle) \mid 1\leq i \leq {\rm Krull.dim}(R/I) \},
    \]
    which is called the stabilized Hilbert regularity of $I$, and moreover
    \[
    \mathrm{max.GB.deg}_{\prec}(I) \leq \overline{\rm hilb}(I)
    \]
    for the graded {reverse} lexicographical ordering $\prec$ on $R$ with $x_n \prec \cdots \prec x_1$.
\end{theorem}

\begin{lemma}[{\cite[Lemma 23]{HS}} and {\cite[Proposition 3.2 (v)]{HSS18}}]\label{lem:KrullDim1Equi}
    If $r = 1$, then the following are all equivalent:
    \begin{enumerate}
        \item[(1)] $I$ is in Noether position with respect to $x_1,\ldots,x_{n-1}$ (or one of the equivalent conditions given in Lemma \ref{lem:NPequi} holds).
        \item[(2)] $I$ is in quasi stable position with respect to the ordered variables $x_1,\ldots,x_n$ (see e.g., \cite[Definition 6]{HS} for definition).
        \item[(3)] $I$ is in strong Noether position.
    \end{enumerate}
\end{lemma}

\begin{proof}
    The equivalence of (1) and (2) is proved in {\cite[Lemma 23]{HS}}, and that of (2) and (3) follows from \cite[Proposition 3.2 (v)]{HSS18}.
\end{proof}

By Theorem \ref{thm:SNPHilb} and Lemma \ref{lem:KrullDim1Equi}, we obtain the following theorem:

\begin{theorem}[{\cite[Proof of Theorem 30]{HS}}]\label{thm:HS}
Assume that $r=1$.
If $R/( I + \langle x_n \rangle)$ is Artinian (or one of the equivalent conditions given in Lemma \ref{lem:KrullDim1Equi} holds), then we have the inequality \eqref{eq:HS} for the graded reverse lexicographical ordering $\prec$ on $R$ with $x_n \prec \cdots \prec x_1$.
\end{theorem}

{Here, we recall a fact that, if $K$ contains sufficiently many elements, the assumption $\mathrm{Krull.dim}(R/( I + \langle x_n \rangle))=0$ in Theorem \ref{thm:HS} is satisfied after a suitable linear changer of coordinates.
In the rest of this paper, we explain this fact in details, since we could not find any appropriate reference.}

{
Although it may be an easy consequence of the theory of hyperplane sections, there exists a hyperplane \textcolor{black}{defined over} $K$ that does not pass through any of given finite points in a projective space over $\overline{K}$ if $K$ is large enough.
Here, we prove a somehow detailed argument:}

\begin{lemma}\label{lem:finitePoints}
    {{Let $\bm{p}_1,\ldots,\bm{p}_N$ be mutually distinct points} in $\mathbb{P}^{n-1}(\overline{K})$.
    If the cardinality of $K$ is greater than $N$, then there exists a linear form $\ell \in R$ that does not vanish 
    at any of {$\bm{p}_1,\ldots,\bm{p}_N$}.}
\end{lemma}

\begin{proof}
{
For each point {$\bm{p}_k=(p_{1,k}: \cdots : p_{n,k}) \in \mathbb{P}^{n-1}(\overline{K})$}, we denote by $S_k$ the set of all linear forms that vanish at {$\bm{p}_k$}, say
\[
S_{k} = \{ y_1 x_1 + \cdots + y_n x_n \mid y_i \in K \mbox{ and } 
 y_1 p_{1,k} + \cdots {+} y_n p_{n,k} = 0 \},
\]
which is isomorphic, as $K$-linear spaces, to the following $K$-linear space of $K^n$:
\[
S_{k}' = \{ (y_1,\ldots, y_n) \in K^n \mid {p_{1,k}y_1 + \cdots {+}p_{n,k}y_n = 0\}}.
\]
The dimension of $S_{k}'$ as a $K$-linear space is at most $n-1$.
Indeed, for an arbitrary non-zero entry \textcolor{black}{$p_{i,k}$} of {$\bm{p}_k$}, 
the vector with $1$ in the $i$-th entry and $0$'s elsewhere does not belong to $S_k'$, 
so that $S_k'$ is a proper subspace of $K^n$.}

{
\textcolor{black}{
When $K$ is an infinite field, it follows from the irreducibility of an affine space over $K$ that $\bigcup_{k=1}^N S'_k$ cannot cover $K^n$.
When $K$ is a finite field of order $q$, we have}
\[
\textcolor{black}{
\# \left( \bigcup_{k=1}^N S_{k}'\right) \leq \sum_{k=1}^N \# S_k' \leq N \cdot q^{n-1} < q^n},
\]
whence a desired linear form over $K$ exists.}
\end{proof}



\begin{lemma}\label{lem:linear-change}
    {With notation as above, assume that $\mathrm{Krull.dim}(R/I) \leq 1$.
    If the cardinality of $K$ is greater than {$\#V_{\overline{K}}(I)$}, then there exists a linear form $\ell \in R$ such that $\mathrm{Krull.dim}(R/(I+\langle \ell \rangle))=0$.
    Moreover, there also exists a suitable linear coordinate change $\sigma$ sending $\ell$ to $x_n$ such that $\mathrm{Krull.dim}(R/(I^{\sigma}+\langle x_n \rangle))=0$.}
\end{lemma}

\begin{proof}
{
The first part follows from Lemma \ref{lem:finitePoints}.
There exists a $k \in \{1,\ldots,n\}$ such that {the} $x_{k}$-coefficient of $\ell$ is not zero.
We may also assume that the $x_k$-coefficient is $1$, without loss of generality.
Let $\sigma_1$ be the linear transformation of variables $x_1,\ldots,x_{n}$ defined by a permutation exchanging $x_k$ and $x_n$, and write the image 
of $\ell$ by $\sigma_1$ as $\ell^{\sigma_1} = a_0 x_0 + a_1x_1 + \cdots + a_n x_n$ for $a_i \in K$ with $a_{n} = 1$.
Then, let $\sigma_2$ be a linear transformation of variables $x_1,\ldots,x_{n}$ defined by
\[
    h^{\sigma_2} := h(x_1,\ldots, x_{n-1},x_{n}-a_1x_1-\cdots - a_{n-1} x_{n-1})
\]
for $h \in R$.
Then, the invertible transformation $\sigma:=\sigma_2 \circ \sigma_1$ of variables $x_1,\ldots,x_n$ sends $\ell$ to $x_{n}$, so that $\mathrm{Krull.dim}(R/(I^{\sigma}+\langle x_n \rangle))=0$.}
\end{proof}

\begin{remark}\label{rem:surj}
{
In Lemma \ref{lem:linear-change}, we have an exact sequence
\[{
	\xymatrix{
	 (R/I)_{d'-1} \ar[r]^{\times \ell} & (R/I)_{d'} \ar[r] & (R/( I+ \langle \ell \rangle))_{d'} \ar[r] & 0.    \\
	}}
\]
for each $d' \geq 1$.
Therefore, the minimal $d$ for which the multiplication-by-$\ell$ map $(R/I)_{d'-1} \to (R/I)_{d'}$ is surjective for any $d' \geq d$ is equal to $d_{\rm reg}(I+\langle  \ell \rangle)$.}
\end{remark}

\subsection{Generalized semi-regular sequences}

In this subsection, we recall the definition of cryptographic semi-regular sequences and that of generalized ones.
Throughout this subsection, let $f_1, \ldots , f_m \in R$ be {\it homogeneous} polynomials of {\it positive} degrees $d_1, \ldots , d_m$, and put $I = \langle f_1, \ldots , f_m \rangle$.

\begin{theorem}[cf.\ {\cite[Theorem 1]{Diem2}}]\label{lem:Diem2}
For each $d$ with $d \geq \mathrm{max}\{d_i : 1 \leq i \leq m \}$, the following are equivalent:
\begin{enumerate}
    \item[(1)] For each $i$ with $1 \leq i \leq m$, if a homogeneous polynomial $g \in R$ satisfies $g f_i  \in \langle f_1, \ldots , f_{i-1} \rangle_R$ and $\mathrm{deg}(g f_i) < d$, then we have $g \in \langle f_1, \ldots , f_{i-1} \rangle_R$, namely, the multiplication-by-$f_i$ map $(R/\langle f_1,\ldots,f_{i-1}\rangle)_{t-d_i} \longrightarrow (R/\langle f_1,\ldots,f_{i-1}\rangle)_{t}$ is injective for any $t$ with $d_i \leq t < d$.
    \item[(2)] We have
    \begin{equation}\label{eq:dregHil}
    {\rm HS}_{R/I}(z) \equiv \frac{\prod_{j=1}^{m}(1-z^{d_j})}{(1- z)^n} \pmod{z^d},
    \end{equation}
    where $\mathrm{HS}$ denotes the Hilbert series.
    \item[(3)] We have $H_1 (K_{\bullet}(f_1, \ldots , f_m))_{\leq d-1} = 0$, where $K_{\bullet}(f_1,\ldots,f_m)$ denotes the Koszul complex on the sequence $(f_1,\ldots,f_m)$.
\end{enumerate}
\end{theorem}

\begin{definition}[{\cite[Definition 3]{BFS}}; see also {\cite[Definition 1]{Diem2}}]\label{def:semiregB}
    For each $d$ with $d \geq \mathrm{max}\{d_i : 1 \leq i \leq m \}$, a sequence $(f_1,\ldots, f_m)$ of homogeneous polynomials in $R\smallsetminus K$ is said to be $d$-regular if the equivalent conditions in Theorem \ref{lem:Diem2} hold.
\end{definition}


Here, we define a cryptographic semi-regular sequence:

\begin{definition}[{\cite[Definition 5]{BFS}}, {\cite[Definition 5]{BFSY}}; see also {\cite[Section 2]{Diem2}}]\label{def:csemireg}
A sequence $(f_1, \ldots , f_m)$ of homogeneous polynomials in $R\smallsetminus K$ is said to be {\it cryptographic semi-regular} if it is $d_{\rm reg}(I)$-regular.
\end{definition}

The Hilbert series of $R/I$ for a cryptographic semi-regular sequence $(f_1,\ldots,f_m)$ is easily computed, by the following Proposition \ref{prop:Diem}:

\begin{proposition}[{\cite[Proposition 1 (d)]{Diem2}}; see also {\cite[Proposition 6]{BFSY}}]\label{prop:Diem}
With the same notation as in Theorem \ref{lem:Diem2}, a sequence $(f_1, \ldots , f_m)$ of homogeneous polynomials in $R\smallsetminus K$ is \textcolor{black}{cryptographic} semi-regular if and only if
\begin{equation}\label{eq:semiregHil2}
    {\rm HS}_{R/\langle f_1,\ldots,f_m\rangle}(z) = \left[ \frac{\prod_{j=1}^{m}(1-z^{d_j})}{(1- z)^n} \right],
\end{equation}
where $[\cdot]$ means truncating a formal power series over $\mathbb{Z}$ after the last consecutive positive coefficient.
(Note that, in this case, we have $d_{\rm reg}(I)=\deg({\rm HS}_{R/I}(z))+1$.)
\end{proposition}

The above `standard' form \eqref{eq:semiregHil2} of Hilbert series was introduced by Fr\"{o}berg in his celebrated paper~\cite{Froberg}.
We will recall his famous conjecture (Conjecture \ref{conj:Froberg} below) in the next subsection.

\begin{remark}
    {In \cite[Definition 1.3]{FH94}, Fr\"{o}berg-Hollman already defined a \textit{sufficiently generic ideal} as one generated by homogeneous polynomials $f_1,\ldots,f_m\in R$ satisfying \eqref{eq:semiregHil2}.
    Namely, a cryptographic semi-regular sequence is nothing but a sequence of homogeneous polynomials generating a sufficiently generic ideal.}
\end{remark}

Here, we recall the following lemma {by Fr\"{o}berg}:

\begin{lemma}[{\cite[b) on page 121]{Froberg}}]\label{lem:Hil_lex}
With notation as above, we have
    \[
    {\rm HS}_{R/I}(z) \geq^{\rm lex} \left[ \frac{\prod_{j=1}^{m}(1-z^{d_j})}{(1- z)^n} \right]
    \]
in the lexicographic sense, i.e., the first coefficient that differs, if any, is greater on the left hand side, see \cite[page 118]{Froberg}.
\end{lemma}

\begin{remark}
For any non-negative integer $d$, it follows from Lemma \ref{lem:Hil_lex} that
\[
\left( {\rm HS}_{R/I}(z) \bmod{z^d} \right)\geq^{\rm lex} \left( \frac{\prod_{j=1}^{m}(1-z^{d_j})}{(1- z)^n} \bmod{z^d} \right).
\]
\end{remark}

To deal with the case where $R/I$ has Krull dimension one, we extend the notion of \textcolor{black}{cryptographic} semi-regular as follows:

\begin{definition}[{\cite[Definition 2.3.3]{KY24b}}]\label{def:gensemireg}
    A sequence $(f_1, \ldots , f_m)$ of homogeneous polynomials in $R\smallsetminus K$ is said to be {\it generalized cryptographic semi-regular} if it is $\widetilde{d}_{\rm reg}(I)$-regular.
\end{definition}

\begin{remark}[cf.\ {\cite[Section 4.1]{KY24b}}]\label{rem:construction}
    A generalized cryptographic semi-regular (but neither regular nor semi-regular) sequence is easily generated.
    For example, choosing coefficients of $g_1,\ldots,g_m$ with $m\geq n$ uniformly at random from a finite field $K=\mathbb{F}_q$ and choosing an arbitrary point $\bm{p}=(p_1,\ldots,p_{n-1},p_n) \in \mathbb{A}^n(K)$ with $p_n=1$, then it is expected that $(f_1,\ldots,f_m)$ with $f_j:= g_j(x_1,\ldots,x_n) - g_j(p_1,\ldots,p_n)x_n^{d_j}$ is generalized cryptographic semi-regular (but neither regular nor semi-regular).
    Without loss of generality, we may take $\bm{o}=(0:\cdots :0 :1)$ as $\bm{p}$, see Section \ref{sec:generic} below for details.
\end{remark}

\begin{example}\label{ex:gcsr}
Let $K=\mathbb{F}_{31}$, $n=3$, and $R=K[x_1,x_2,x_3]$.
Here we consider the case where $(d_1,d_2,d_3)=(3,2,2)$ with $m=3$, and let 
    \begin{eqnarray*}
        f_1 &=& 27 x_1^3 + 5 x_1^2 x_2 + 6 x_1 x_2^2 + 20 x_2^3 + 29 x_1 x_2 x_3 + 25 x_2^2 x_3 + 3 x_1 x_3^2 + 26 x_2 x_3^2,\\
        f_2 &=& 2 x_1^3 + 29x_1^2 x_2 + 25 x_1 x_2^2 + 3 x_2^3 + 25 x_1^2 x_3 + 20 x_1 x_2 x_3 + 25 x_2^2 x_3 + 19x_1x_3^2 + 24x_3 x_3^2,\\
        f_3 &=& 12 x_1^2 + 12x_1x_2 + 23x_2^2 + x_1 x_3 + 10x_2 x_3.
    \end{eqnarray*}
Putting $I=\langle f_1,f_2,f_3\rangle$, one can verify that
    \[
    {\rm HS}_{R/I}(z) = 1 + 3 z + 5 z^2 + 5 z^3 + 3 z^4  + z^5 + z^6 + \cdots
    \]
with $\mathrm{Krull.dim}(R/I)=1$, $D:=\mathrm{hilb}(I)=5$, and $\dim_K (R/I)_D =1$, so that $(f_1,f_2,f_3)$ is $D$-regular (in fact $D+1$-regular).
Namely $(f_1,f_2,f_3)$ is generalized cryptographic semi-regular, but is neither regular nor cryptographic semi-regular.
\end{example}

\begin{lemma}\label{lem:dreg_bound}
If the sequence $\bm{F} = (f_1,\ldots, f_m)$ is generalized cryptographic semi-regular, then we have
\begin{equation}\label{eq:easy}
    \widetilde{d}_{\rm reg}(I)\leq {\deg\left(\left[\frac{\prod_{j=1}^m(1-z^{d_j})}{(1-z)^{n}}\right]\right)}+1.
\end{equation}
\end{lemma}

\begin{proof}
By our assumption that $\bm{F}$ is generalized cryptographic semi-regular, the equality \eqref{eq:dregHil} holds for $d = \widetilde{d}_{\rm reg}(I)$.
For any $d'$ with $d' < \widetilde{d}_{\rm reg}(I)$, we have ${\rm HF}_{R/I}(d') > 0$, whence the $z^{d'}$-coefficient of {$(\prod_{j=1}^m(1-z^{d_j}))/(1-z)^{n}$} is also positive.
Thus, the inequality \eqref{eq:easy} holds.
\end{proof}


For a sequence $\bm{F}=(f_1,\ldots,f_m)$ of homogeneous polynomials in $R \smallsetminus K$, we consider the following four conditions:
\begin{description}
    \item[{\bf (A)}] $\bm{F}$ is regular, equivalently {${\rm HS}_{R/I}(z) = (\prod_{j=1}^{m}(1-z^{d_j}))/(1- z)^n$}.
    \item[{\bf (B)}] $\bm{F}$ is cryptographic semi-regular (i.e., $d_{\rm reg}(I)$-regular).
    \item[{\bf (C)}] $\bm{F}$ is generalized cryptographic semi-regular (i.e., $\widetilde{d}_{\rm reg}(I)$-regular).
    \item[{\bf (D)}] $\bm{F}$ is $\mathrm{hilb}(I)$-regular.
\end{description}

\begin{lemma}\label{lem:four}
    For conditions (A), (B), (C), and (D) as defined above, there is an implication:
    \[
    (A) \Rightarrow (B) \Rightarrow (C) \Rightarrow (D), \mbox{ but } (D) \overset{r\geq 2}{\not\Rightarrow} (C) \overset{r=1}{\not\Rightarrow} (B) \overset{m>n}{\not\Rightarrow} (A).
    \]
    Namely, (A) and (B) are equivalent for $m \leq n$, (B) and (C) are equivalent for $r \neq 1$, (C) and (D) are equivalent for $r \leq 1$, and (A), (B), (C), and (D) are all equivalent for $m = n$ and $r =0$.
\end{lemma}

\begin{proof}
    It is clear that $(A) \Rightarrow (B) \Rightarrow (C) \Rightarrow (D)$.
    When $r \leq 1$, it follows from \eqref{eq:Hilb} that $(D) \Rightarrow (C)$ holds.
    Also when $r \neq 1$, we have $d_{\rm reg}(I) = \widetilde{d}_{\rm reg}(I)$, whence $(C) \Rightarrow (B)$ holds.
    For $m \leq n$, Proposition \ref{prop:Diem} implies $(B) \Rightarrow (C)$.

    {On the other hand, it is clear that (B) does not imply (A) if $m > n$.
    The sequence $(f_1,f_2,f_3)$ provided in Example \ref{ex:gcsr} gives a counterexample for the implication (C) $\Rightarrow$ (B) when $r=1$.
    As for (D) $\Rightarrow$ (C), here is a counterexample with $r=2$:
    Let $(g_1,g_2):=(x^3,xy) \in \mathbb{Q}[x,y,z]^2$ and put $I:=\langle g_1,g_2\rangle$.
    Then one can verify that $\mathrm{HS}_{R/I}(z) = 1+3z+5z^2 + 6z^3 + 7z^4 + \cdots$ \textcolor{black}{with} $\mathrm{Krull.dim}(R/I)=2$ and $D:=\mathrm{hilb}(I)=2$, whence $(g_1,g_2)$ is $D$-regular but not regular.}
\end{proof}

    




\subsection{Genericness {and Fr\"{o}berg's conjecture}}\label{subsec:generic}

Given $m$ {positive} integers $d_1,\ldots,d_m$, 
let ${V_{n,m,(d_1,\ldots,d_m)}} := R_{d_1} \times \cdots \times R_{d_m}$, which is topologically identified with 
the affine space {$\mathbb{A}^{N_1}(K)\times \dots \times \mathbb{A}^{N_m}(K) = \mathbb{A}^N(K)$ of 
dimension $N:=\sum_{j=1}^mN_j$} with $N_j:= \binom{n+d_j-1}{d_j}$.
{Here, we identify a form in $R_{d_j}$ with a point in $\mathbb{A}^{N_j}(K)$ by taking the coefficients.}
{We may write $K^N$ simply for $\mathbb{A}^N(K)$.}
%
{We also correspond a point $\bm{F}=(f_1,\ldots,f_m) \in V_{n,m,(d_1,\ldots,d_m)}$ to an ideal $\langle \bm{F} \rangle_R :=\langle f_1,\ldots,f_m\rangle_R$, where the subscript $R$ may be omitted if the meaning is clear from the context.}

{
Following Fr\"{o}berg's definition (cf.\ \cite{Froberg}, \cite{Froberg22}, \cite{FS18}), we say that a form of $R_d$ is {\it generic} if all monomials of degree $d$ occur with nonzero coefficient and if all the coefficients are algebraically independent over the prime field of $K$.
We call $\bm{F}=(f_1,\ldots,f_m) \in V_{n,m,(d_1,\ldots,d_m)}$ a \textit{generic} sequence if all $f_i$'s are generic forms with coefficients that are mutually algebraic independent over the prime field.
If $K$ does not contain $N$ elements that are algebraically independent over the prime field of $K$ (for example $K$ is an algebraic extension of the prime field), we may replace $K$ by its arbitrary extension field $L$ containing such $N$ elements.
When we consider a property of Hilbert series, this makes a sense since every Hilbert series is not changed by extending the field of coefficients.}

{
On the other hand, the genericness of a property for $\bm{F}\in V_{n,m,(d_1,\ldots,d_m)}$ is defined as follows:}

\begin{definition}[cf.\ {\cite[Section 1]{Pardue}}]\label{def:generic}
    For given $n$, $m$, and $(d_1,\ldots,d_m)$, a property $\mathcal{P}$ for $\bm{F}=(f_1,\ldots,f_m)\in V_{n,m,(d_1,\ldots,d_m)}$ is said to be {\it generic} if it holds on a non-empty Zariski open set in $V_{n,m,(d_1,\ldots,d_m)}$, namely there exists a non-empty Zariski open set $U$ in $V_{n,m,(d_1,\ldots,d_m)}$ such that
    \[
    U \subset \{ \bm{F}=(f_1,\ldots,f_m) \in V_{n,m,(d_1,\ldots,d_m)} : \mbox{$\bm{F}$ satisfies $\mathcal{P}$} \}.
    \]
\end{definition}

{
Here, we recall Fr\"{o}berg's important lemma for the minimality of the Hilbert series on generic sequences.
Note that the characteristic of $K$ is assumed to be $0$ in the original paper~\cite{Froberg}, but it can be easily proved that the assertion of the lemma holds for any characteristic.
We here state the lemma in a slightly modified (but equivalent) form for the convenience of our use.
}

\begin{lemma}[{\cite[Lemma 1]{Froberg}}]\label{lem:parametric-Hilbert}
    {
    With notation as above, let $\Omega$ be an extension field of $K$, and 
    let {$\mathcal{C} := \{ c_{j,t} : 1 \leq j \leq m,\ t\in \mathcal{T}_{d_j}\}$} 
    be a set of elements in $\Omega$ that are algebraically independent over the prime field $K_0$ of $K$. 
    {Here $\mathcal{T}_d$ is the set of monomials of degree $d$ in $\{x_1,\ldots,x_n\}$.} 
    Let $L:=K(\mathcal{C})$, let $R_L := R \otimes_K L = L[x_1,\ldots,x_n]$, and put {$g_j:= \sum_{t\in \mathcal{T}_{d_j}} c_{j,t}t \in (R_L)_{d_j}$} for each $1\leq j \leq m$.
    Namely $\bm{G}=(g_1,\ldots,g_m)$ is a {\em generic sequence}.
    Then, for any $\bm{F}=(f_1,\ldots,f_m)\in V_{n,m,(d_1,\ldots,d_m)}$, we have the coefficient-wise inequality}
    \begin{equation}
        \mathrm{HS}_{R/\langle f_1,\ldots,f_m\rangle_R}(z) = \mathrm{HS}_{R_L/\langle f_1,\ldots, f_m \rangle_{R_L}} \geq \mathrm{HS}_{R_L/\langle g_1,\ldots,g_m \rangle_{R_L}}(z).
    \end{equation}
\end{lemma}


\begin{remark}\label{rem:Diem}
{
Since every Hilbert series is not changed by extending the field of coefficients, the Hilbert series of the quotient by an ideal generated by a generic sequence depends only on $\mathrm{char}(K)$, $n$, $m$, and $(d_1,\ldots,d_m)$, see e.g., \cite[Lemma and Definition 4]{Diem}.}    
Thus, 
in discussing the Hilbert series, we may consider the ideal generated by $\bm{G}$ over $K_0(\mathcal{C})[X]$ instead of that over $K(\mathcal{C})[X]$.
{Because, the reduced Gr\"obner basis of $\langle \bm{G}\rangle_{R_L}$ coincides with  
that of \textcolor{black}{$\langle \bm{G}\rangle_{K_0(\mathcal{C})[X]}$}, 
since the well-known Buchberger's algorithm can be performed over a field, 
over which the generating set is defined. Hence 
the minimal generating set of ${\rm in}(\langle \bm{G}\rangle_{R_L})$ 
coincides with that of ${\rm in}(\langle \bm{G}\rangle_{K_0(\mathcal{C})[X]})$.
This also implies that, for computing the reduced Gr\"obner basis of 
$\langle \bm{G}\rangle_{R_L}$, we may assume that all elements of $\mathcal{C}$ 
are algebraically independent over $K$ and compute it over $L=K(\mathcal{C})$.
}
\end{remark}

{
If $K$ is an infinite field, Fr\"{o}berg-L\"{o}fwall proved in \cite[Theorem 1]{FL90} that there exists a non-empty Zariski-open subset {$U_{\rm min} \subset V_{n,m,(d_1,\ldots,d_m)}$} on which the quotients by the corresponding ideals have the same Hilbert series, which is minimal in the coefficient-wise sense.
Note that, even if $K$ contains the set $\mathcal{C}$ of $N$ algebraically independent elements, $\bm{F} \in U_{\rm min}$ not necessarily implies the genericness of $\bm{F}$.
}

{
On the other hand, the Hilbert series of $R/\langle \bm{F} \rangle$ for $\bm{F}\in U_{\rm min}$ is equal to that of $R_L/\langle \bm{G}\rangle$ for a generic sequence $\bm{G}$ over $L$ as in Lemma \ref{lem:parametric-Hilbert}.
\textcolor{black}{The existence of the non-empty Zariski open set $U_{\rm min}$} can be easily verified by considering the {\em stability} of the reduced Gr\"obner basis 
in the context of so-called a {\em comprehensive {Gr\"obner system} for a parametric ideal}. 
{A} result by Nabeshima in \cite{Nabe} on ``generic Gr\"obner basis of a parametric ideal'' can be translated directly to our problem 
where the generic sequence $\bm{G}$ is given from $K[\mathcal{C}][x_1,\ldots,x_n]$ and 
$\mathcal{C}$ is the set of {\em coefficient parameters}. 
{Here we assume that all elements of $\mathcal{C}$ are algebraically independent over $K$ and 
hence they are treated as independent indeterminates. (See Remark \ref{rem:Diem}.)}
In the below, we show a useful result for our purpose.
We let $\prec_1$ be a block {ordering}
with $X:=\{x_1,\ldots,x_n\}\succ\succ \mathcal{C}$ such that 
its restriction on $X$ coincides with the given {ordering} $\prec$. 
Let $\bm{H}$ be a finite set of $K[\mathcal{C}][X]$ and 
${\rm GB}(\bm{H})$ the reduced Gr\"obner basis of the ideal $\langle \bm{H}\rangle_{L[x]}$ 
with respect to the restriction $\prec$ of $\prec_1$ on the monomials in $X$, 
{where $L=K(\mathcal{C})$}. 
Moreover, let $\widetilde{\rm GB}(\bm{H})
:=\{{\rm dlcm}(h)h: h\in {\rm GB}(\bm{H})\}$ and 
$\ell(\bm{H})=\prod_{h'\in \widetilde{\rm GB}(\bm{H})} {\rm LC}_{\prec}(h')$ in 
$K[\mathcal{C}][X]$, 
where ${\rm dlcm}(h)$ denotes the least common multiple of all denominators of coefficients in 
$K(\mathcal{C})$ of $h$ and each {$h'\in \widetilde{\rm GB}(\bm{H})$} is considered as an element in $K[\mathcal{C}][X]$, 
that is, a polynomial in $X$ over $K[\mathcal{C}]$. 
For $\bm{a}\in \mathbb{A}^N(K)$, we denote by $\pi_{\bm{a}}$ the substitution homomorphism 
from $K[\mathcal{C}][X]$ to $R=K[X]$, that is, $\pi_{\bm{a}}(h(\mathcal{C},X))=h(\bm{a},X)$ 
for $h \in K[\mathcal{C},X]$. For $h\in L[X]$, when the denominator of each coefficient of $h$ 
does not vanish at $\bm{a}$, {the image} $\pi_{\bm{a}}(h)$ is {defined} naturally. 
}
%

{
\begin{theorem}[{\cite[Theorem 3]{Nabe}}]\label{thm:parametric}
With notation as above, consider $\bm{H}$ and $\widetilde{\rm GB}(\bm{H})$ to be subsets of $K[\mathcal{C}\cup X]$, and 
let $S(\bm{H})$ be the reduced Gr\"obner basis of the ideal quotient 
$\langle \bm{H}\rangle_{K[\mathcal{C}\cup X]}:\langle \widetilde{\rm GB}(\bm{H})\rangle_{K[\mathcal{C}\cup X]}$ with respect to $\prec_1$.
Then, {we have $S(\bm{H})\cap K[\mathcal{C}]\not=\emptyset$}, and 
for any $\bm{a}$ in $\mathbb{A}^N(\overline{K}) \smallsetminus (V_{\overline{K}}(S(\bm{H})\cap K[\mathcal{C}])\cup V_{\overline{K}}(\ell(\bm{H})))$, the set
$\pi_{\bm{a}}({\rm GB}(\bm{H}))$ is the reduced Gr\"obner basis of 
$\langle \pi_{\bm{a}}(\bm{H})\rangle_{K[X]}$ with respect to $\prec$.
\end{theorem}
}

{
Now we return to our case where $\bm{H}=\bm{G}$.  
Let \textcolor{black}{$\mathcal{O}=\mathbb{A}^N(K)\smallsetminus  (V_{K}(S(\bm{G})\cap K[\mathcal{C}])\cup V_{K}(\ell(\bm{G})))$}. Then $\mathcal{O}$ is 
Zariski-open in $\mathbb{A}^N(K)$. Moreover, $\mathcal{O}$ is not empty, 
since \textcolor{black}{$\langle S(\bm{G})\cap K[C]\rangle\cap \langle \ell(\bm{G})\rangle$} is a non-zero ideal and 
since $K$ is an infinite field. 
It is obvious that for any $\bm{a} \in \mathcal{O}$, $\pi_{\bm{a}}({\rm GB}(\bm{G}))$ is the reduced Gr\"obner basis 
of $\langle \pi_{\bm{a}}(\bm{G})\rangle_R$ with respect to $\prec$. Moreover,  for such $\bm{a}$,  
since ${\rm LC}_{\prec}(\pi_{\bm{a}}(h))=1$ for any $h\in {\rm GB}(\bm{G})$, we have 
${\rm LM}_{\prec}(\pi_{\bm{a}}({\rm GB}(\bm{G})))={\rm LM}_{\prec}({\rm GB}(\bm{G}))$. 
}

{
\begin{corollary}\label{cor:parametric}
Under the assumption in Theorem \ref{thm:parametric}, 
there is a non-empty Zariski open set $\mathcal{O}$ in $K^N$ such that, for any $\bm{a}$ in $\mathcal{O}$, 
the set of leading monomials of the reduced Gr\"obner basis of $\langle \pi_{\bm{a}}(\bm{G})\rangle_R$ 
coincides with that of $\langle \bm{G}\rangle_{R_L}$. 
\end{corollary}

For each ideal $I$ of $K[X]$, 
the \textcolor{black}{Hilbert} series ${\rm HS}_{R/I}(z)$ is determined by its initial ideal and so by the leading monomials of its (reduced) Gr\"obner basis.
Thus, Corollary \ref{cor:parametric} implies that ${\rm HS}_{R/\langle \pi_{\bm{a}}(\bm{G})\rangle}(z)$ coincides with ${\rm HS}_{R_L/\langle \bm{G}\rangle}(z)$ for any $\bm{a}\in \mathcal{O}$. 
Considering the corresponding set $U$ 
of $\mathcal{O}$ in {$V_{n,m,(d_1,\ldots,d_m)}$}, 
we have the following proposition and a conjecture. 
}
%
%
{Here, we remark that 
the following proposition may be a well-known fact.
However, the authors could not find any reference for its proof, and therefore they would write a proof here for the readers' convenience; the idea of the proof will be used to obtain an analogous result in Section \ref{sec:generic} below.}

\begin{proposition}\label{prop:generic_equivalence}
{
    Assume that $K$ is an infinite field.
    Then, for given $n$, $m$, and $(d_1,\ldots,d_m)$, the following conditions (1), (2), (3), and (4) are equivalent to each other:
    \begin{enumerate}
        \item[(1)] There exists an {$\bm{F}=(f_1,\ldots,f_m)\in V_{n,m,(d_1,\ldots,d_m)}$} satisfying \eqref{eq:semiregHil2}, i.e., $\bm{F}$ is cryptographic semi-regular (in other words, the ideal $\langle \bm{F}\rangle$ is sufficiently generic).
        \item[(2)] For the non-empty Zariski open set $U_{\rm min}$ defined in the paragraph just after Remark \ref{rem:Diem}, any $\bm{F} =(f_1,\ldots,f_m)\in U_{\rm min}$ satisfies \eqref{eq:semiregHil2}.
        \item[(3)] The property that {$\bm{F}=(f_1,\ldots,f_m)\in V_{n,m,(d_1,\ldots,d_m)}$} satisfies \eqref{eq:semiregHil2} is generic (in terms of Definition \ref{def:generic}).
        \item[(4)] For $\mathcal{C}$, $L$, $R_L$, and $\bm{G}=(g_1,\ldots,g_m)$ as in Lemma \ref{lem:parametric-Hilbert}, we have
        \begin{equation}\label{eq:Hil-generic}
                    \mathrm{HS}_{R_L/\langle g_1,\ldots,g_m \rangle}(z) = \left[ \frac{\prod_{j=1}^m(1-z^{d_j})}{(1- z)^n}  \right].
        \end{equation}
        Namely, the generic sequence $\bm{G}$ is cryptographic semi-regular (in other words, the ideal $\langle \bm{G} \rangle_{R_L}$ is sufficiently generic).
    \end{enumerate}}
\end{proposition}

\begin{proof}
    {
    The implication (2) $\Rightarrow$ (1) is obvious.}

    {
    To see (1) $\Rightarrow$ (4), it follows from Lemma \ref{lem:Hil_lex} and Lemma \ref{lem:parametric-Hilbert} that
    \begin{equation}\label{eq:ineq}
         \mathrm{HS}_{R/\langle f_1,\ldots,f_m\rangle}(z) \geq \mathrm{HS}_{R_L/\langle g_1,\ldots,g_m \rangle}(z) \geq^{\rm lex}  \left[ 
        \frac{\prod_{j=1}^m(1-z^{d_j})}{(1- z)^n}
        \right].
    \end{equation}
    By our assumption that $\mathrm{HS}_{R/\langle f_1,\ldots,f_m\rangle}(z)$ satisfies \eqref{eq:semiregHil2}, we have \eqref{eq:Hil-generic}.}

    {The implication (4) $\Rightarrow$ (3) follows from Corollary \ref{cor:parametric}.}

    {The implication (3) $\Rightarrow$ (2) follows from the irreducibility of {$V_{n,m,(d_1,\ldots,d_m)}$}.}
\end{proof}

Here, we recall {\it Fr\"{o}berg's conjecture}:

\begin{conjecture}[Fr\"{o}berg's conjecture~\cite{Froberg}]\label{conj:Froberg}
    Assume that $K$ is an infinite field.
    For {any} $n$, $m$, and $(d_1,\ldots,d_m)$,
    {the equivalent conditions (1), (2), (3), and (4) in Proposition \ref{prop:generic_equivalence} hold.}
\end{conjecture}

See Pardue's paper \cite{Pardue} for some conjectures equivalent to Fr\"{o}berg's one (Conjecture \ref{conj:Froberg}).
He also proved that Moreno-Soc\'{i}as conjecture~\cite{MS} is stronger than Conjecture \ref{conj:Froberg}, see \cite[Theorem 2]{Pardue} for a proof.

\section{Proofs of Theorems \ref{thm:mainT} and \ref{thm:WRL-GB}}

{As in the previous sections, let} $R=K[x_1,\ldots,x_n]$ be the polynomial ring of $n$ variables over a field $K$ (in any characteristic){.
Let} $I$ be a proper homogeneous ideal of $R${, and let} $F = \{ f_1,\ldots,f_m \}$ be a {homogeneous} subset of $R$ generating $I$.
For each $1\leq j \leq m$, put $d_j := \deg(f_j)$ (total degree).


\subsection{Proof of Theorem \ref{thm:mainT}}\label{subsec:general}

{We first prove Theorem \ref{thm:mainT} {(1)}, which can be easily proved by Lemma \ref{lem:dreg_bound} together with the following easy lemma in the case where $m=n-1$:}

\begin{lemma}\label{lem:m=n}
    With notation as above, if $m = n-1$ and if the sequence $(f_1,\ldots,f_m,x_n)$ is cryptographic semi-regular (equivalently, it is regular since $m+1=n$), then we have $\widetilde{d}_{\rm reg}(I) = d_{\rm reg}(I+\langle x_n \rangle_{R})-1$, and hence
    \[
    \mathrm{max.GB.deg}_{\prec}(I) \leq \mathrm{max}\{ d_{\rm reg}(I + \langle x_n \rangle), \widetilde{d}_{\rm reg}(I) \} = d_{\rm reg}(I+\langle x_n \rangle)= \sum_{j=1}^{n-1} (d_j-1)+1
    \]
    for the graded reverse lexicographical ordering $\prec$ with $x_n \prec \cdots \prec x_1$.
    Moreover, $(f_1,\ldots,f_m)$ is generalized cryptographic semi-regular (equivalently, it is regular since $m<n$).
\end{lemma}

\begin{proof}
{Since $(f_1,\ldots,f_m,x_n)$ is regular, $(f_1,\ldots,f_m)$ is also regular, so that
\[
{\rm HS}_{R/I} (z) = {\rm HS}_{R/(I+\langle x_n \rangle)}(z) \cdot \frac{1}{1-z} = {\rm HS}_{R/(I+\langle x_n \rangle)}(z) \cdot (1+z+z^2 + \cdots)
\]
with
\[
    {\rm HS}_{R/(I+\langle x_n \rangle)}(z) = \frac{\prod_{j=1}^{n-1}(1-z^{d_j}) \cdot (1-z)}{(1-z)^{n}} = \frac{\prod_{j=1}^{n-1}(1-z^{d_j})}{(1-z)^{n-1}} = \prod_{j=1}^{n-1}\left( \sum_{k=0}^{d_j-1}z^k\right),
\]
which is a polynomial.}
{It is clear that \textcolor{black}{$\widetilde{d}_{\rm reg}(I) = d_{\rm reg}(I+\langle x_n \rangle_{R})-1$}, as desired.}
\end{proof}

\noindent {\it Proof of Theorem \ref{thm:mainT} {(1)}:}
{If $m \geq n$, our assumption that $(f_1,\ldots,f_m,x_n)$ is cryptographic semi-regular implies}
\[
\textcolor{black}{
d_{\rm reg}(I+\langle x_n \rangle ) = \deg\left(\left[\frac{\prod_{j=1}^m(1-z^{d_j})}{(1-z)^{n-1}}\right]\right)+1,}
\]
{whose right hand side is upper-bounded by $D^{(n,m)}$ for $m \geq n$.}
{On the other hand, since $(f_1,\ldots,f_m)$ is generalized cryptographic semi-regular, it follows from Lemma \ref{lem:dreg_bound} that we have the inequality \eqref{eq:easy},} the right hand side of which is {nothing but} {$D^{(n,m)}$} for $m\geq n$.
{Therefore, we obtain \eqref{eq:max_Dnew}.}

{If $m=n-1$, it} follows from Lemma \ref{lem:m=n} that \textcolor{black}{$\mathrm{max}\{ d_{\rm reg}(I + \langle x_n \rangle), \widetilde{d}_{\rm reg}(I) \} $} is upper-bounded by $\sum_{j=1}^{n-1} (d_j-1)+1$, which is equal to {$D^{(n,m)}$} for $m=n-1$.\qed

\bigskip


{Next, we shall prove (2) of Theorem \ref{thm:mainT}.}
For this, we estimate the complexity of computing a Gr\"{o}bner basis of a homogeneous ideal with Macaulay matrices.

{Let} $\prec$ be an {\it arbitrary} graded monomial ordering on $R$.
{It is well-known that the} complexity of the Gr\"{o}bner basis computation for {the homogeneous set} $F$ is estimated as that of computing the reduced row echelon form (RREF) of the degree-$D$ {\it Macaulay matrix}, denoted by $M_{\leq D}(F)$, where $D$ is an arbitrary upper-bound on $\mathrm{max.GB.deg}_{\prec}(F)$.
Here, the degree-$D$ Macaulay matrix $M_{\leq D}(F)$ (in homogeneous case) is a block matrix of the form 
\[
M_{\leq D}(F) = \begin{pmatrix}
M_{D}(F) &  & \\
& \ddots &\\
& & M_{d_0}(F)
\end{pmatrix}
\]
with $d_0:=\min \{ \deg (f_j) : 1 \leq j \leq m \}$, where each $M_d(F)$ is a {\it homogeneous} Macaulay matrix defined as follows:
For each $d$ with $d \geq d_0$, the degree-$d$ homogeneous Macaulay matrix $M_d(F)$ of $F$ has columns indexed by the {monomials} of $R_d$ sorted, from left to right, according to the chosen order $\prec$.
The rows of $M_d(F)$ are indexed by the polynomials $m_{i,j}f_j$, where $m_{i,j} \in R$ is a {monomial} such that $\deg(m_{i,j}f_j ) = d$.
\textcolor{black}{For a monomial $t\in \mathcal{T}_d$ of degree $d$, the $(m_{i,j}f_j, t)$-entry of $M_d(F)$ is defined as the coefficient of $t$ in $m_{i,j}f_j$.}

\begin{theorem}\label{thm:complexity}
    With notation as above, we have the following:
    \begin{enumerate}
        \item[(1)] {\cite[Proposition 1]{BFS-F5}} The RREF of $M_{\leq D}(F)$ is computed in
    \begin{equation}\label{eq:DGB_complexity0}
        O\left(m D \binom{n+D-1}{D}^{\omega} \right).
    \end{equation}
    Hence, if $\prec$ is a graded \textcolor{black}{ordering}, then the complexity of the Gr\"{o}bner basis computation for $F$ is upper-bounded {by \eqref{eq:DGB_complexity0} for} $D=\mathrm{max.GB.deg}_{\prec}(F) $.
    \item[(2)] Let $D$ be a non-negative integer such that $D = O(n)$.
    Then, the RREF of $M_{\leq D}(F)$ is computed in
    \begin{equation}\label{eq:DGB_complexity}
        O\left(m \binom{n+D-1}{D}^{\omega} \right).
    \end{equation}
    Hence, if $\prec$ is a graded \textcolor{black}{ordering}, and if $\mathrm{max.GB.deg}_{\prec}(F) = D=O(n)$, then the complexity of the Gr\"{o}bner basis computation for $F$ is upper-bounded by \eqref{eq:DGB_complexity}.
    \end{enumerate}
\end{theorem}

\begin{proof}
It suffices to estimate the complexity of computing the RREFs of the homogeneous Macaulay matrices $M_{D}(F), \ldots , M_{d_0}(F)$.
For each $d$ with $d_0 \leq d \leq D$, the number of rows (resp.\ columns) in $M_d(F)$ is $\displaystyle \sum_{1\leq i \leq m,\ d \geq d_i} \binom{n-1+d-d_i}{d-d_i}$ (resp.\ $\binom{n-1+d}{d}$).
It follows from $d - d_i \leq d - 1$ that
\[
\sum_{1\leq i \leq m,\ d \geq d_i} \binom{n-1+d-d_i}{d-d_i} \leq \sum_{i=1}^m \binom{n-1+d-1}{d-1}= m \binom{n-1+d-1}{d-1},
\]
whence the number of rows in $M_d(F)$ is upper-bounded by $m \binom{n-1+d-1}{d-1}$.
Therefore, the complexity of the row reductions on $M_D(F),\ldots,M_{d_0}(F)$ is upper-bounded by
\begin{align}
    \sum_{d=d_0}^D m \binom{n-1+d-1}{d-1} \binom{n-1+d}{d}^{\omega-1} & \leq m \sum_{d=1}^D \binom{n-1+d-1}{d-1} \binom{n-1+d}{d}^{\omega-1} \nonumber \\
    & \leq m \sum_{d=1}^D \binom{n-1+d}{d}^{\omega}, \label{eq:complexity_est}
\end{align}
where we used a fact that the reduced row echelon form of a $k \times \ell$ matrix $A$ over a field $K$ can be computed in $O(k \ell^{\omega-1})$, see Remark \ref{rem:RREF} below for its proof.
Clearly \eqref{eq:complexity_est} is upper-bounded by $m D\binom{n+D-1}{D}^{\omega}$, we obtain the estimation \eqref{eq:DGB_complexity0} in (1).

In the following, we shall improve the estimation \eqref{eq:DGB_complexity0} in (1) to obtain \eqref{eq:DGB_complexity} in (2), assuming that $D=O(n)$.
Here, for each $d$ with $1\leq d\leq D-1$, we have
\[
\binom{n-1+d}{d}=\frac{(n-1+d)!}{(n-1)!d!}=
\binom{n-1+d+1}{d+1}\times \frac{d+1}{n+d}.
\]
Repeating this procedure, we obtain
\begin{align*}
&\binom{n-1+d+1}{d+1}\times \frac{d+1}{n+d} = 
\binom{n-1+d+2}{d+2} \times \frac{d+2}{n+d+1} \times \frac{d+1}{n+d}\\
&= \cdots = \binom{n-1+D}{D} \times \frac{D}{n+D-1} \times \cdots \times \frac{d+2}{n+d+1} \times \frac{d+1}{n+d}.
\end{align*}
It is clear that $\frac{d'+1}{n+d'}=1-\frac{n-1}{n+d'}$ is monotonically increasing with respect to $d'$. 
Putting $c := (D-1)/n$, we have $D = cn + 1$, so that
\[
\frac{d'+1}{n+d'} \leq \frac{D}{n+D-1} = \frac{cn+1}{(c+1)n} = \frac{c+\frac{1}{n}}{c+1}
\]
for any $d'$ with $d' \leq D-1$.
As we consider $n \to \infty$, we may suppose that 
\[
\frac{d'+1}{n+d'} \leq \frac{c+\frac{1}{n}}{c+1} < \frac{c+\varepsilon}{c+1}<1
\]
for some constant $\varepsilon$ with $0 < \varepsilon < 1$.
Therefore, we obtain
\[
 \binom{n-1+D}{D} \times \frac{D}{n+D-1} \times \cdots \times \frac{d+2}{n+d+1} \times \frac{d+1}{n+d} < \binom{n-1+D}{D}\left( \frac{c+\varepsilon}{c+1} \right)^{D-d},
\]
whence
\begin{align*}
    \sum_{d=1}^D \binom{n-1+d}{d}^{\omega}&< \binom{n-1+D}{D}^{\omega}\sum_{k=0}^{D-1} \left(\frac{c+\varepsilon}{c+1}\right)^{k\omega} \\
    &= \binom{n-1+D}{D}^{\omega} \times \frac{1-\left(\frac{c+\varepsilon}{c+1} \right)^{D\omega}}{1-\left( \frac{c+\varepsilon}{c+1}\right)^{\omega}}\\
&< \binom{n-1+D}{D}^{\omega} \times \frac{1}{1-\left(\frac{c+\varepsilon}{c+1}\right)^{\omega}},
\end{align*}
where we put $k := D-d$.
It follows from $D = O(n)$ that $c$ is upper-bounded by a constant.
\textcolor{black}{By $0 < \varepsilon<1$, it is clear that $\frac{c+\varepsilon}{c+1}$ is also upper-bounded by a positive constant smaller than $1$, whence $1/{(1-(\frac{c+\varepsilon}{c+1})^{\omega})} = O(1)$ as $n \to \infty$.}
We have proved the assertion (2).
\end{proof}


\begin{remark}\label{rem:RREF}
In the proof of Theorem \ref{thm:complexity}, we used a fact that the reduced row echelon form of a $k \times \ell$ matrix $A$ over a field $K$ can be computed in $O(k \ell^{\omega-1})$, where $k$ can be greater than $\ell$.
    This fact is proved by considering the following procedures to compute the reduced row echelon form of $A$:
    \begin{enumerate}
        \item[(0)] Write $k = \ell q + r$ for $q, r \in \mathbb{Z}$ with $0 \leq r < \ell$.
        \item[(1)] Choose and remove $2 \ell$ rows from $A$, and let $A'$ be a $2\ell \times \ell$ matrix whose rows are the chosen $2\ell $ rows.
        \item[(2)] Compute the reduced row echelon form of $A'$:
        Inserting $\ell$ zeros to the end of each row of $A'$, construct a $2\ell \times 2\ell$ square matrix, and compute its reduced row echelon form.
        \item[(3)] As a result of (2), we obtain more than or equal to $\ell$ zero row vectors.
        We add the nonzero (linearly independent) row vectors obtained to $A$.
        \item[(4)] Go back to (1).
    \end{enumerate}
    Repeating the procedures from (1) to (4) at most $q$ times, we obtain the reduced row echelon form $B$ of the initial $A$.
    Since the complexity of (2) is $(2\ell)^{\omega}$, the total complexity to obtain $B$ is $q \cdot (2\ell)^{\omega} \leq (k/\ell) \cdot 2^{\omega} \ell^{\omega} = O(k\ell^{\omega-1})$, as desired.
\end{remark}

\noindent {\it Proof of Theorem \ref{thm:mainT} {(2)}:}
For the graded reverse lexicographical ordering $\prec$ on $R$ with $x_n \prec\cdots \prec x_1$, it follows from Theorem \ref{thm:mainT} {(1)} that {$D^{(n,m)}$} gives an upper bound on $\mathrm{max.GB.deg}_{\prec}(F)$.
Therefore, a Gr\"{o}bner basis of $\langle F \rangle_{R}$ is obtained by computing the RREF of $M_{\leq D}(F)$ for $D = D^{(n,m)}$.
Here, the value $D^{(n,m)}$ of \eqref{eq:newbound} is maximized at $m=n$:
In this case, $D^{(n,m)}$ is equal to the Lazard's bound $\sum_{j=1}^{n}(d_j-1)+1$, which is $O(n)$ for fixed $d_1,\ldots, d_m$.
Therefore, it follows from Theorem \ref{thm:complexity} (2) that the RREF of $M_{\leq D}(F)$ for $D = D^{(n,m)}$ is computed in the complexity \eqref{eq:new_comp_bound}, as desired.\qed

\medskip

\begin{example}\label{ex:optimal}
The bound $D^{(n,m)}$ in \eqref{eq:max_Dnew} is optimal for a homogeneous ideal $I$ satisfying the assumptions of Theorem \ref{thm:mainT}. 
For example, considering the case where $K=\mathbb{F}_{31}$ and $(n,m)=(3,4)$, and letting $F=\{f_1,f_2,f_3,f_4\}$ be a generating set for $I$ with
\begin{eqnarray*}
    f_1 &=& 3 x_1^2 + 20 x_1 x_2 + x_2^2 + 7 x_1 x_3 + 15x_2 x_3,\\
    f_2 &=& 18 x_1^2 + 15 x_1 x_2 + 14 x_2^2 + 24 x_1 x_3 + 4x_2x_3,\\
    f_3 &=& 6 x_1^2 + 16 x_1 x_2 + x_2^2 + x_1 x_3 + 8 x_2 x_3,\\
    f_4 &=& 21 x_1^2 + x_1 x_2 + 21 x_2^2 + 24 x_1 x_3 + 6x_2 x_3,
\end{eqnarray*}
we have $\mathrm{Krull.dim}(R/I)=1$ and
\begin{eqnarray*}
    \mathrm{HS}_{R/(I+\langle x_3\rangle)}(z)= 1 + 2z,\quad \mathrm{HS}_{R/I}(z) = 1+3z+2z^2+z^3+z^4+z^5+\cdots
\end{eqnarray*}
with $d_{\rm reg}(I+\langle x_3 \rangle) = 2$ and $\widetilde{d}_{\rm reg}(I)=3$, and therefore $(f_1,f_2,f_3,f_4,x_3)$ (resp.\ $(f_1,f_2,f_3,f_4)$) is cryptographic semi-regular (resp.\ generalized cryptographic semi-regular).
One can easily verify $D^{(3,4)}=3$.
The reduced Gr\"{o}bner basis of $I$ with respect to the graded reverse lexicographical ordering $\prec$ with $x_3\prec x_2 \prec x_1$ is
\[
\{ x_2 x_3^2, x_1^2 + 27 x_2x_3, x_1x_2 + 14x_2x_3, x_2^2 + 30x_2x_3, x_1x_3 + 26x_2x_3\},
\]
whence $\mathrm{max.GB.deg}_{\prec}(I) = 3$.
Namely, we have
\[
3 = \mathrm{max.GB.deg}_{\prec}(I)  = \max \{ d_{\rm reg}(I+\langle x_n \rangle),\widetilde{d}_{\rm reg}(I)\} = \widetilde{d}_{\rm reg}(I)= D^{(n,m)}.
\]

On the other hand, $\mathrm{max}\{ d_{\rm reg}(I+\langle x_n \rangle), \widetilde{d}_{\rm reg}(I) \}$ in \eqref{eq:max_Dnew} can be strictly less than $D^{(n,m)}$, e.g., when $F:=\{ f_1:=23x_1^2+26x_1x_2+13x_2^2,f_2:=29x_1^2+x_1x_2+x_2^2\}$ with $n=m=2$ and $K=\mathbb{F}_{31}$, a straightforward computation shows that $I=\langle F \rangle$ satisfies the assumptions of Theorem \ref{thm:mainT} with $\mathrm{max.GB.deg}_{\prec}(I)=2$ and $d_{\rm reg}(I+\langle x_2 \rangle)=  \widetilde{d}_{\rm reg}(I) = 2<D^{(2,2)}=3$.
(In this case, we can replace $D^{(n,m)}$ by $D^{(n,m)}-1$ in \eqref{eq:new_comp_bound}).    
\end{example}







\subsection{Proof of Theorem \ref{thm:WRL-GB}}

As in the previous subsection, let $\prec$ be the graded reverse lexicographical ordering on $R=K[x_1,\ldots,x_n]$ with $x_n \prec \cdots \prec x_1$.
We also denote the restriction of $\prec$ on $K[x_1,\ldots,x_{n-1}]$ by the same notation $\prec$.
In the following, we shall prove {Theorem \ref{thm:WRL-GB}}, namely the inequality \eqref{eq:HS} becomes an equality 
if {the initial ideal ${\rm in}_{\prec}(I)=\langle \mathrm{LM}_{\prec}(I)\rangle_R$} is weakly reverse lexicographic.
Here, a weakly reverse lexicographic ideal is a monomial ideal $J$ such that if $x^{\alpha}$ is one of the minimal generators of $J$ then every monomial of the same degree which preceeds $x^{\alpha}$ must belong to $J$ as well (see also e.g., \cite[Section 4]{Pardue} for the definition).

As in \cite{Hashemi2010} and \cite{HS} (see also \cite{KY} and \cite{KY24b}), we consider the elimination ideal
\[
\textcolor{black}{    I' := (I+\langle x_n \rangle_R) \cap K[ x_1,\ldots,x_{n-1} ] = \langle f_1|_{x_n=0},\ldots,f_m|_{x_n = 0} \rangle_{K[x_1,\ldots,x_{n-1}]}.}
\]
Since $R/( I + \langle x_n \rangle)$ and $K[x_1,\ldots,x_{n-1}]/I'$ are isomorphic as graded rings, they have the same Hilbert series, and $d_{\rm reg}(I') = d_{\rm reg}(I+\langle x_n \rangle_R)$.
Put $D := d_{\rm reg}(I') = d_{\rm reg}(I+\langle x_n \rangle_R)$.
By comparing Hilbert series, it can be also easily checked that, the sequence $(f_1,\ldots,f_m,x_n)\in R^{m+1}$ is cryptographic semi-regular if and only if {$(f_1|_{x_n=0},\ldots,f_m|_{x_n=0}) \in K[x_1,\ldots,x_{n-1}]^m$} is cryptographic semi-regular, equivalently these sequences are both $D$-regular.

Let $G$ (resp.\ $G'$) be the reduced Gr\"{o}bner basis of $I$ (resp.\ $I'$) with respect to $\prec$.
Note that, since {$I+\langle x_n \rangle_R = I'R + \langle x_n \rangle_R$}, the reduced Gr\"{o}bner basis of $I+\langle x_n \rangle_R$ is given by $G' \cup \{x_n \}$, whence $\mathrm{max.GB.deg}_{\prec}(I') = \mathrm{max.GB.deg}_{\prec}(I+\langle x_n \rangle_R)$.

To prove {Theorem} \ref{thm:WRL-GB}, we start with extending some results in \cite{KY} to our case.

\begin{lemma}[cf.\ {\cite[Corollary 2]{KY}}]\label{lem:LMideal}
    With notation as above, we assume that $(f_1,\ldots,f_m,x_n)\in R^{m+1}$ is cryptographic semi-regular (i.e., {$(f_1|_{x_n=0},\ldots,f_m|_{x_n=0}) \in K[x_1,\ldots,x_{n-1}]^m$} is cryptographic semi-regular), namely both the sequences are $D$-regular with $D := d_{\rm reg}(I') = d_{\rm reg}(I+\langle x_n \rangle_R)$.
    Then, we have
    \begin{equation}\label{eq:LMideal}
        (\langle \mathrm{LM}_{\prec}(I) \rangle_R)_d = (\langle \mathrm{LM}_{\prec}(I') \rangle_R)_d
    \end{equation}
    for any $d$ with $d < D$.
\end{lemma}

\begin{proof}
We first prove $(\langle \mathrm{LM}_{\prec}(I) \rangle_{R})_d \supset (\langle \mathrm{LM}_{\prec}(I') \rangle_{R})_d $ for any non-negative integer $d$.
For this, it suffices to show that $\mathrm{LM}_{\prec}(f) \in \mathrm{LM}_{\prec}(I)$ for any $f \in I_d' \smallsetminus \{ 0 \}$.
We may assume that $f$ is homogeneous, and then it is written as $f=g + x_n h$ for some $g \in I_d$ and $h \in R_{d-1}$.
Here, $g$ is not divisible by $x_n$, since otherwise $f \notin K[x_1,\ldots,x_{n-1}]$, a contradiction.
Therefore, we have $\mathrm{LM}_{\prec}(f) = \mathrm{LM}_{\prec}(g)$ with $g \in I$, as desired.

    Since $(f_1,\ldots,f_m,x_n)$ is $D$-regular, it follows from \cite[Proposition 2 a)]{Diem2} (cf.\ \cite[Corollary 1]{KY}), its sub-sequence $(f_1,\ldots,f_m)$ is also $D$-regular.
    Hence, we have
    \begin{align*}
            {\rm HS}_{R/I} (z)  & \equiv  \frac{\prod_{j=1}^m(1-z^{d_j})}{(1-z)^{n}} \equiv  \frac{\prod_{j=1}^m(1-z^{d_j}) \cdot (1-z)}{(1-z)^{n}} \cdot \frac{1}{1-z} \\
            & \equiv {\rm HS}_{R/(I+\langle x_n \rangle)}(z) \cdot \frac{1}{1-z} \equiv {\rm HS}_{R/( I+\langle x_n \rangle)}(z) \cdot (1+z+z^2 + \cdots) \\
            &\equiv {\rm HS}_{K[x_1,\ldots,x_{n-1}]/ I'}(z) \cdot (1+z+z^2 + \cdots) \pmod{z^D}.
    \end{align*}
    Thus, for any $d$ with $d < D  = d_{\rm reg}(I')$, one has ${\rm HF}_{R/I}(d) = \sum_{i=0}^d \mathrm{HF}_{K[x_1,\ldots,x_{n-1}]/I'}(i)$, so that
        \begin{equation}\label{eq:dim_eq}
        \dim_K R_d - \dim_K I_d = \sum_{i=0}^d (\dim_K K[x_1,\ldots,x_{n-1}]_i - \dim_K I'_i) .
        \end{equation}
    Since $(JR)_d = x_n^{d} J_0 \oplus x_n^{d-1} J_1 \oplus \cdots \oplus x_n J_{d-1} \oplus J_d$ for any of $J=I'$ and $J=K[x_1,\ldots,x_{n-1}]$ (note that this holds also for $d \geq D$), the right hand side of \eqref{eq:dim_eq} is equal to $\dim_K R_d - \dim_K (I'R)_d$, so that $\dim_K (I'R)_d = \dim_K I_d$.
    Here, it follows from \cite[Theorem 5.2.6]{Singular} that $\dim_K (\langle \mathrm{LM}_{\prec}(I) \rangle_R)_d = \dim_K (\langle \mathrm{LM}_{\prec}(I'R) \rangle_R)_d$.
    Since $\langle \mathrm{LM}_{\prec}(I'R) \rangle_R = \langle \mathrm{LM}_{\prec}(I') \rangle_R$, we finally obtain \eqref{eq:LMideal}.
\end{proof}

\begin{lemma}[cf.\ {\cite[Lemma 2]{KY}}]\label{lem:LM}
Under the same setting as in Lemma \ref{lem:LMideal}, for each degree $d<D$, we have
\begin{equation}\label{eq:LMD}
{\rm LM}_{\prec}(G)_{d} ={\rm LM}_{\prec}(G')_{d}.
\end{equation}
\end{lemma}

\begin{proof}
{This can be proved just by translating the proof of \cite[Lemma 2]{KY}:
Replace $R'$, $R$, $\langle \bm{F}^h \rangle_{R'}$, $\langle \bm{F}^{\rm top} \rangle_R$, $G_{\rm hom}$, and $G_{\rm top}$ in the proof by $R$, $K[x_1,\ldots,x_{n-1}]$, $I$, $I'$, $G$, and $G'$ in our case, respectively.}
\end{proof}

\begin{lemma}\label{lm:maxGB}
    Under the same setting as in Lemma \ref{lem:LMideal}, we also suppose that $R/(I+\langle x_n \rangle)$ is Artinian (equivalently $K[x_1,\ldots,x_{n-1}]/I'$ is Artinian), namely $D < \infty$.
    Furthermore, we assume that $\langle \mathrm{LM}_{\prec}(I)\rangle_R$ is weakly reverse lexicographic.
    Then we have $\mathrm{max.GB.deg}_{\prec}(I)  \geq  D$ and $\mathrm{max.GB.deg}_{\prec}(I') = D$.
\end{lemma}

\begin{proof}
     Put $d := \mathrm{max.GB.deg}_{\prec}(I)$.
     Assume for a contradiction that $ d <  D$.
     Since $\mathrm{Krull.dim}(R/(I+\langle x_n \rangle)) = 0$, it follows from Lemma \ref{lem:NPequi} that there is an element $g$ in \textcolor{black}{$G$} with ${\rm LM}_{\prec}(g)=x_{n-1}^{d'}$ for some $d'\leq d$.
     Then, any monomial $t\in R_{d'}$ with $t\succeq x_{n-1}^{d'}$ (namely $t$ is any monomial in $K[x_1,\ldots,x_{n-1}]_{d'}$) belongs to 
     {the initial ideal ${\rm in}_{\prec}(I)=\langle {\rm LM}_{\prec}(I)\rangle_{R}=\langle {\rm LM}_{\prec}(G)\rangle_{R}$} by the weakly reverse lexicographicness.
     Here, by $d<D$ together with \eqref{eq:LMD}, we have $\langle {\rm LM}_{\prec}(G)\rangle_{R}=\langle {\rm LM}_{\prec}(G')_{\leq d} \rangle_{R}$.
     This implies that $t$ also belongs to $\langle {\rm LM}_{\prec}(G')\rangle_{K[x_1,\ldots,x_{n-1}]}$, and hence $K[x_1,\ldots,x_{n-1}]_{d'}/I'_{d'}=0$.
     Thus, we have $d\geq d'\geq d_{\rm reg}(I')$, a contradiction.

     By the same argument as above, we can show that $\mathrm{max.GB.deg}_{\prec}(I') = d_{\rm reg}(I')$ as follows.
     Note that $\mathrm{max.GB.deg}_{\prec}(I') \leq d_{\rm reg}(I')$ follows from a general fact written in the second paragraph of Remark \ref{rem:general} below.
     Thus, we \textcolor{black}{suppose} to the contrary that $d:=\mathrm{max.GB.deg}_{\prec}(I') < D:=d_{\rm reg}(I')$. Then, it follows from \eqref{eq:LMD} that ${\rm LM}_{\prec}(G')={\rm LM}_{\prec}(G)_{\leq d}$\textcolor{black}{, and} 
     {the initial ideal ${\rm in}_{\prec}(\langle G'\rangle)=\langle {\rm LM}_{\prec}(G')\rangle_{K[x_1,\ldots,x_{n-1}]}$} 
     has the {\em weak reverse lexicographicness up to $d$}. 
     Since there is an element $g$ in $G'$ with ${\rm LM}_{\prec}(g)=x_{n-1}^{d'}$  for some $d'\leq d$, any monomial in $K[x_1,\ldots,x_{n-1}]_{d'}$ belongs to $\langle {\rm LM}_{\prec}(G')\rangle_{K[x_1,\ldots,x_{n-1}]}$ and so $K[x_1,\ldots,x_{n-1}]_{d'}/I'_{d'}=0$, which implies $d\geq d'\geq D$, a contradiction.
\end{proof}

\begin{remark}\label{rem:general}
    In Lemma \ref{lm:maxGB}, when $R/(I+\langle x_n \rangle)$ is Artinian (or equivalently $K[x_1,\ldots,x_{n-1}]/I'$ is Artinian), we can easily prove $\mathrm{max.GB.deg}_{\prec}(I') = D$ if 
    {${\rm in}_{\prec}(I')=\langle \mathrm{LM}_{\prec}(I' )\rangle_{K[x_1,\ldots,x_{n-1}]}$} 
    is a weakly reverse lexicographic ideal, assuming neither that $(f_1,\ldots,f_m,x_n)$ is cryptographic semi-regular nor that 
    {${\rm in}_{\prec}(I)=\langle \mathrm{LM}_{\prec}(I )\rangle_R$} is weakly reverse lexicographic.
    
    In general, for any homogeneous ideal $J\subset R$ with $\mathrm{Krull.dim}(R/J)=0$, we have $\mathrm{max.GB.deg}_{\prec}(J) \leq \mathrm{reg}(J) = d_{\rm reg}(J)$ by \cite[Corollary 2.5]{BS} (see also \cite[Proposition 4.15]{Hashemi2010}), and it is straightforward that the equality holds if 
    {the initial ideal $\langle \mathrm{LM}_{\prec}(J )\rangle_R$ of $J$} is a weakly reverse lexicographic.
\end{remark}


\noindent {\it Proof of {Theorem \ref{thm:WRL-GB}}:}
    Put $D := d_{\rm reg}(I+\langle x_n \rangle_R) = d_{\rm reg}(I')$ and $D' := \widetilde{d}_{\rm reg}(I)$.
    From our assumption that $R/(I+\langle x_n \rangle)$ is Artinian, we have the inequality \eqref{eq:HS}, i.e., $\mathrm{max.GB.deg}(I) \leq \mathrm{max}\{ D,D'\}$.
    {The case where $\mathrm{Krull.dim}(R/I) = 0$ is clear from Remark \ref{rem:general}, so we assume that $\mathrm{Krull.dim}(R/I) = 1$.}
    Note that, for the reduced Gr\"{o}bner basis $G$ of $I$, the elements of $\mathrm{LM}_{\prec}(G)$ are the minimal generators of the monomial ideal $\langle \mathrm{LM}_{\prec}(I) \rangle_R = \langle \mathrm{LM}_{\prec}(G) \rangle_R$.

    First we consider the case where $D' \ge D$, so that $\max \{ D,D'\} = D'$.
    Put $d := \mathrm{max.GB.deg}_{\prec} ( I ) {\leq D'}$.
    Then, it follows from {Macaulay's basis theorem} (cf.\ \cite[Theorem 1.5.7]{KR1}) that $R_d/ I_{d}$ has a $K$-linear basis of the form
    \[
    \{t\in R_d : \mbox{$t$ is a monomial and } {t\notin \langle {\rm LM}_{\prec}(I)\rangle_R = \langle \mathrm{LM}_{\prec}(G)\rangle_R }\},
    \]
    which is called the {\it standard monomial basis}.
    We set $r := \dim_K R_d/I_d$, and write this standard monomial basis as $\{ t_1, \ldots, t_r \}$ {with $t_1 \succ \cdots \succ t_r$}.
    In the following, we prove that $\{ t_1 x_n^{s} ,\ldots, t_r x_n^s\}$ is a $K$-linear basis of $R_{d+s}/I_{d+s}$ for any $s$ with $s \geq 1$, 
    from which we have $\dim_K R_d/I_d=\dim_K R_{d+s}/I_{d+s}$ for any positive integer $s$, so that $d\geq D'$ and therefore $d = D'$.

    By Lemma \ref{lm:maxGB}, we have $d \geq D$, and thus it \textcolor{black}{follows from Remark \ref{rem:surj}} that the multiplication-by-$x_n^s$ map from $R_d/I_d$ to $R_{d+s}/I_{d+s}$ is surjective (since $R/(I+\langle x_n \rangle)$ is Artinian).
    Therefore {the set} $B_s:=\{t_1x_n^s,\ldots,t_rx_n^s\}$ generates $R_{d+s}/I_{d+s}$ {over $K$}.
    Suppose to \textcolor{black}{the} contrary that $B_s$ is not a basis. 
    In this case, there exists an $i$ such that $t_i x_n^s$ is divisible by $\mathrm{LM}_{\prec}(g)$ for some $g \in G$.
    Indeed, since the elements of $B_s$ are linearly dependent over $K$, we obtain $f:=a_1 t_1 x_n^s + \cdots + a_r t_r x_n^s \in I_{d+s}$ for some $(a_1,\ldots,a_r) \in K^r \smallsetminus \{ (0,\ldots,0)\}$.
    Thus, for the least $i$ with $a_i \neq 0$, the leading monomial of $f$ is $t_i x_n^s$, which is divisible by $\mathrm{LM}_{\prec}(g)$ for some $g \in G$.
    Putting $u_i= \mathrm{GCD}(t_i,\mathrm{LM}_{\prec}(g))$ and $s_i=t_i/u_i$, we have $u_i s_i x_n^s = t_i x_n^s$.
    Since $t_i = s_i u_i$ is not divisible by $\mathrm{LM}_{\prec}(g)$ {and since $u_i= \mathrm{GCD}(t_i,\mathrm{LM}_{\prec}(g))$}, we can write $u_i x_n^{s'} = \mathrm{LM}_{\prec}(g)$ for some $s'\leq s$. 
    (We note that $\deg s_i=d-\deg u_i\geq \deg {\rm LM}_{\prec}(g)-\deg u_i=s'$.) 
    Note that $s' \geq 1$ since otherwise $t_i$ is divisible by $\mathrm{LM}_{\prec}(g)$.

    Take an arbitrary monomial $s'_i$ such that $\deg s'_i=s'$ and $s_i'$ divides $s_i$. 
    Then, {by ${\rm LM}_{\prec}(g)=u_ix_n^{s'} \preceq u_is'_i$ together with the weakly reverse lexicographicness of} {the initial ideal ${\rm in}_{\prec}(I)=\langle \mathrm{LM}_{\prec}(I) \rangle_R$,} the monomial $s'_iu_i$ should belong to $\langle {\rm LM}_{\prec}(G)\rangle_R$. 
    Moreover, since $s'_iu_i$ divides $t_i=s_iu_i$, the monomial $t_i=s_iu_i$ also belongs to $\langle {\rm LM}_{\prec}(G)\rangle_R$, which is a contradiction. 

    Finally, consider the remaining case, i.e., $D>D'$, so that $\max \{ D,D' \} =D$.
    In this case, it follows from Lemma \ref{lm:maxGB} that $d \geq D$, whence $d = D$.\qed

\section{{Results on genericness with conjectures}}\label{sec:generic}

{
Let $n$ be a positive integer, and let $R = K[x_1,\ldots,x_n]$ be the polynomial ring of $n$ variables over a field $K$.
We may assume that $K$ is infinite.
Let $m$, $d_1,\ldots, d_m$ be positive integers.
As in \textcolor{black}{Section} \ref{subsec:generic}, let $V_{n,m,(d_1,\ldots,d_m)} = R_{d_1} \times \cdots \times R_{d_m}$, and naturally endow it with the Zariski topology.
{Recall Definition \ref{def:generic} for the definition of the genericness of a property \textcolor{black}{for} $\bm{F}=(f_1,\ldots,f_m)\in V_{n,m,(d_1,\ldots,d_m)}$.}
{For $\bm{F}=(f_1,\ldots,f_m)\in V_{n,m,(d_1,\ldots,d_m)}$, we denote by $\langle \bm{F}\rangle_R$ (or by $\langle \bm{F} \rangle$ simply) the ideal generated by $f_1,\ldots,f_m$.}
}

\subsection{Restricted genericness and our conjectures}\label{subsec:rest_generic}

{
For $n$, $m$, $(d_1,\ldots,d_m)$, and a non-negative integer $k$, we set
\begin{align*}
    \Lambda_{n,m,(d_1,\ldots,d_m)}^{{\rm dim}=k} &:= \{ \bm{F} \in V_{n,m,(d_1,\ldots,d_m)} : \mathrm{Krull.dim}(R/\langle \bm{F}\rangle)=k \},\\
    \Lambda_{n,m,(d_1,\ldots,d_m)}^{{\rm dim}\geq k} &:= \{ \bm{F}\in V_{n,m,(d_1,\ldots,d_m)} : \mathrm{Krull.dim}(R/\langle \bm{F}\rangle)\geq k \},\\
    \Lambda_{n,m,(d_1,\ldots,d_m)}^{\rm csr} &:= \{ {\bm{F}} \in V_{n,m,(d_1,\ldots,d_m)} : \mbox{$\bm{F}$ is cryptographic semi-regular} \},\\
    \Lambda_{n,m,(d_1,\ldots,d_m)}^{\rm gcsr} &:= \{ {\bm{F}} \in V_{n,m,(d_1,\ldots,d_m)} : \mbox{$\bm{F}$ is generalized cryptographic semi-regular} \}.
\end{align*}
{It follows from Lemma \ref{lem:four} that $\Lambda_{n,m,(d_1,\ldots,d_m)}^{\rm gcsr} \cap \Lambda_{n,m,(d_1,\ldots,d_m)}^{{\rm dim}=0} \subset \Lambda_{n,m,(d_1,\ldots,d_m)}^{\rm csr}$ and the equality holds if and only if $m \geq n$}.}

Here, as a variant of Conjecture \ref{conj:Froberg}, we {start with raising} the following conjecture (cf.\ \cite[Conjecture 4.3.4]{KY24b} as its strict version):

\begin{conjecture}\label{conj:ours}
    Assume that $K$ is an infinite field.
    For {any} $n$, $m$, and $(d_1,\ldots,d_m)$, 
the property that $\bm{F}=(f_1,\ldots,f_m)\in V_{n,m,(d_1,\ldots,d_m)}$ is generalized cryptographic semi-regular is generic.
\end{conjecture}

We see that Conjectures \ref{conj:Froberg} and \ref{conj:ours} are equivalent to each other.
For this, we first prove the following lemma:

\begin{lemma}\label{lem:Art}
    For {any $n$, $m$, and $(d_1,\ldots,d_m)$ with} $m \geq n$, the property that $R/\langle f_1,\ldots,f_m \rangle$ is Artinian is generic.
\end{lemma}

\begin{proof}
    {
    The case $m=n$ is well-known to hold:
    More strongly, $\Lambda^{{\rm dim}=0}_{n,n,(d_1,\ldots,d_n)}$ is a non-empty Zariski-open set of $V_{n,n,(d_1,\ldots,d_n)}$.}

    {
    Assume $m > n$, fix $n$, $m$, and $(d_1,\ldots,d_m)$.
    Then, the preimage of $\Lambda^{{\rm dim}=0}_{n,n,(d_1,\ldots,d_n)}$ by a natural projection sending an element of $V_{n,m,(d_1,\ldots,d_m)}$ to the first $n$-entries is included in $\Lambda^{{\rm dim}=0}_{n,m,(d_1,\ldots,d_m)}$.
    Since $\pi$ is continuous, the preimage is Zariski-open in $V_{n,m,(d_1,\ldots,d_m)}$, and it is clearly non-empty, whence the assertion holds.}
\end{proof}


\begin{proposition}\label{prop:generic}
    {Assume that $K$ is an infinite field.
    For $m \geq n$, Conjecture \ref{conj:ours} is equivalent to Fr\"{o}berg's conjecture (Conjecture \ref{conj:Froberg}).}
\end{proposition}

\begin{proof}
{Assume that there exists a non-empty Zariski-open set $U_1 \subset V_{n,m,(d_1,\ldots,d_m)}$ with $U_1 \subset \Lambda^{\rm gcsr}_{n,m,(d_1,\ldots,d_m)}$.
    By Lemma \ref{lem:Art}, there exists a non-empty Zariski-open set $U_2 \subset V_{n,m,(d_1,\ldots,d_m)}$ such that $U_2 \subset \Lambda^{{\rm dim}=0}_{n,m,(d_1,\ldots,d_m)}$.
    Since both $U_1$ and $U_2$ are dense, we have $\emptyset \neq U_1\cap U_2 \subset \Lambda^{\rm gcsr}_{n,m,(d_1,\ldots,d_m)} \cap \Lambda^{{\rm dim}=0}_{n,m,(d_1,\ldots,d_m)} \subset \Lambda^{\rm csr}_{n,m,(d_1,\ldots,d_m)}$, {whence Fr\"{o}berg's conjecture holds for $(n,m,d_1,\ldots,d_m)$}.}

    {
    The converse is obvious by $\Lambda^{\rm csr}_{n,m,(d_1,\ldots,d_m)} \subset \Lambda^{\rm gcsr}_{n,m,(d_1,\ldots,d_m)}$ {(cf.\ Lemma \ref{lem:four})}.}
\end{proof}

Despite of the equivalence between two conjectures, 
those cannot guarantee {that {$\Lambda_{n,m,(d_1,\ldots,d_m)}^{\rm gcsr} \cap \Lambda_{n,m,(d_1,\ldots,d_m)}^{{\rm dim}\geq 1}$} includes a non-empty open set in the subspace {$\Lambda_{n,m,(d_1,\ldots,d_m)}^{{\rm dim}\geq 1}$} equipped with the induced topology.
To discuss this, we define the {\it restricted genericness} as follows:}


\begin{definition}[Restricted genericness]\label{def:restricted-generic}
    {Let $V'$ be a non-empty subset of $V_{n,m,(d_1,\ldots,d_m)}$.
    For given $n$, $m$, and $(d_1,\ldots,d_m)$, a property $\mathcal{P}$ for {$\bm{F}\in V_{n,m,(d_1,\ldots,d_m)}$} is said to be {{\it generic on $V'$} if} it holds on a non-empty open set of $V'$ with respect to the induced topology, namely there exists a non-empty Zariski open set $U$ in $V_{n,m,(d_1,\ldots,d_m)}$ such that {$\emptyset \neq U \cap V' \subset \{ \bm{F} \in V_{n,m,(d_1,\ldots,d_m)} : \mbox{$\bm{F}$ satisfies $\mathcal{P}$} \}$}.}
\end{definition}

{
Namely, it is important whether the property that $\bm{F} \in V_{n,m,(d_1,\ldots,d_m)}$ is generalized cryptographic semi-regular is {\it generic on} {$\Lambda_{n,m,(d_1,\ldots,d_m)}^{{\rm dim}\geq 1}$} or not.}
\textcolor{black}{We will also discuss the genericness on
\begin{equation}\label{eq:Zp}
    Z_{\bm{p}}:=\{ \bm{H} \in V_{n,m,(d_1,\ldots,d_m)} : \bm{p}\in V_{\overline{K}}(\bm{H}) \}
\end{equation}
for each point $\bm{p} \in \mathbb{P}^{n-1}(K)$.}

\begin{remark}\label{cor:generic}
{If $K$ is an infinite field and if the property that $\bm{F} \in V_{n,m,(d_1,\ldots,d_m)}$ is generalized cryptographic semi-regular is {\it generic on} {$\Lambda_{n,m,(d_1,\ldots,d_m)}^{{\rm dim}\geq 1}$},}
then `almost' all sequences {$\bm{F} \in V_{n,m,(d_1,\ldots,d_m)}$ such that $\mathrm{Krull.dim}(R/\langle \bm{F} \rangle) \geq 1$} are expected to be generalized cryptographic semi-regular.
\end{remark}

{
As the first discussion about the \textcolor{black}{restricted} {genericness} of {\textcolor{black}{the generalized cryptographic semi-regularity} on $ \Lambda_{n,m,(d_1,\ldots,d_m)}^{{\rm dim}\geq 1}$, 
we consider the following sub-case, 
where $\mathrm{Krull.dim}(R/\langle \bm{F} \rangle) = 1$ and $\dim_K R/\langle \bm{F}\rangle_{D}=1$ for $D=\widetilde{d}_{\rm reg}(\langle \bm{F} \rangle)<\infty$.
In this case, $\langle \bm{F}\rangle$ has \textcolor{black}{a} unique projective 
zero $\bm{p}$ in $\mathbb{P}^{n-1}(K)$}. 
First we show that 
it suffices to consider the case of the point $\bm{o}:=(0:\cdots:0:1)$ \textcolor{black}{as $\bm{p}$}, by the following lemma:
}

\begin{lemma}\label{lem:other_point}
    {If the property that $\bm{F}\in V_{n,m,(d_1,\ldots,d_m)}$ is generalized cryptographic semi-regular is generic on $Z_{\bm{o}}$, then it is also generic on {$Z_{\bm{p}}$ for any $\bm{p} \in \mathbb{P}^{n-1}(K)$}.}
\end{lemma}

\begin{proof}
    For each {$\bm{p}=(p_1 : \ldots : p_n) \in \mathbb{P}^{n-1}(K)$}, let $\phi_{\bm{p}}$ be an invertible projective linear transformation sending $\bm{p}$ to $\bm{o}$.
    We may assume {$p_n \neq 0$ and so $p_n=1$} 
    without loss of generality, so that we can take $\phi_{\bm{p}}$ as
    \[
{(x_1:\cdots : x_n) \to (x_1-p_1x_n:\cdots:x_{n-1}-p_1x_{n-1}:x_{n})}.
    \]
    Then, the induced polynomial map $\phi_{\bm{p}}^{\ast} : R \to R $ is an isomorphism substituting 
    {$x_i-p_ix_n$} 
    to $x_i$ in $f\in R$ for $1\leq i \leq n-1$ and stabilizing $x_n$.
%
Here, we consider a map on $V_{n,m,(d_1,\ldots,d_m)}$ defined by
    \[
    \Phi_{\bm{p}} : V_{n,m,(d_1,\ldots,d_m)} \to V_{n,m,(d_1,\ldots,d_m)} \ ; \ (f_1,\ldots,f_m) \mapsto (\phi_{\bm{p}}^{\ast}(f_1),\ldots,\phi_{\bm{p}}^\ast(f_m)).
    \]
It is straightforward that this is an invertible linear (and hence homeomorphic) map on $V_{n,m,(d_1,\ldots,d_m)}$.
{Moreover, it is easily checked that 
$\phi_{\bm{p}}^*$ does not change LM for any $h\in R$ with respect to the 
graded reverse lexicographical ordering $\prec$ such that $x_1\succ \cdots \succ x_n$. 
Thus, for any subset $S$ of $R$, we have ${\rm LM}(S)={\rm LM}(\phi_{\bm{p}}^*(S))$, from which 
${\rm in}_\prec(\langle \bm{F}\rangle)={\rm in}_\prec(\langle \Phi_{\bm{p}}(\bm{F})\rangle)$ 
for $\bm{F}$ in $V_{n,m,(d_1,\ldots,d_m)}$.
This also implies that ${\rm HS}_{R/\langle \bm{F}\rangle}(z)={\rm HS}_{R/{\rm in}(\langle \bm{F}\rangle)}(z)
={\rm HS}_{R/{\rm in}(\langle \Phi_{\bm{p}}(\bm{F})\rangle)}(z)=
{\rm HS}_{R/\langle \Phi_{\bm{p}}(\bm{F})\rangle}(z)$. 
}

From our assumption, we have a non-empty Zariski-open set $U_{\bm{o}}$ of $V_{n,m,(d_1,\ldots,d_m)}$ 
intersecting with $Z_{\bm{o}}$ such that $U_{\bm{o}} \cap Z_{\bm{o}} \subset \Lambda_{n,m,(d_1,\ldots,d_m)}^{\rm gcsr}\cap Z_{\bm{o}}$.
    Then $\Phi_{\bm{p}}(U_{\bm{o}})$ is a non-empty Zariski-open set of $V_{n,m,(d_1,\ldots,d_m)}$ intersecting with $\Phi_{\bm{p}}(Z_{\bm{o}})=Z_{\bm{p}}$.
{
Moreover, as any $\bm{F}\in U_{\bm{o}}\cap Z_{\bm{o}}$ is generalized \textcolor{black}{cryptographic} semi-regular, 
$\Phi_{\bm{p}}(\bm{F})$ is also generalized cryptographic semi-regular, 
{as 
${\rm HS}_{R/\langle \bm{F}\rangle}(z)={\rm HS}_{R/\langle \Phi_{\bm{p}}(\bm{F})\rangle}(z)$}. 
Thus, letting $U_{\bm{p}}=\Phi_{\bm{p}}(U_{\bm{o}})$,  we have 
$U_{\bm{p}} \cap Z_{\bm{p}} \subset \Lambda_{n,m,(d_1,\ldots,d_m)}^{\rm gcsr}\cap Z_{\bm{p}}$.}
\end{proof}

Note that
\[
Z_{\bm{o}} = \{ (f_1,\ldots,f_m) \in V_{n,m,(d_1,\ldots,d_m)} : \mbox{the $x_n^{d_j}$-coefficient of $f_j$ is zero ($1\leq \forall j \leq m$)}\}.
\]

\begin{definition}[Generic sequence for $Z_{\bm{o}}$]\label{df:genericsequence}
{
For checking the genericness of the generalized cryptographic semi-regularity {of a sequence} in $Z_{\bm{o}}$, we introduce a {\em generic sequence $(g_1,\ldots,g_m)$} \textcolor{black}{for $Z_{\bm{o}}$} with $\deg(g_j)=d_j$ as follows:
Let $\Omega$ be an extension field of $K$, and 
let $\mathcal{C}:= \{ c_{j,t} : 1\leq j\leq m,\ t \in \mathcal{T}_{d_j}\smallsetminus \{x_n^{d_j}\} \}$ be a 
set of elements in $\Omega$ that are algebraically independent over the prime field $K_0$ of $K$. 
Then we set $L:=K(\mathcal{C})$ and $L_0:=K_0(\mathcal{C})$. 
Putting $g_j=\sum_{t\in \mathcal{T}_{d_j}\smallsetminus \{x_n^{d_j}\}} c_{j,t}t$ for each $1\leq j \leq m$, 
we call a sequence $\bm{G}=(g_1,\ldots,d_m)$ in $K_0[\mathcal{C}][X]\subset L_0[X]\subset L[X]$ a {generic sequence} \textcolor{black}{for $Z_{\bm{o}}$}. 
We also set $\mathcal{F}_j:=\left\{\sum_{t\in \mathcal{T}_{d_j}\smallsetminus \{x_n^{d_j}\}} a_{j,t}t: a_{j,t}\in K\right\}$ for $1\leq j \leq m$, so that $Z_{\bm{o}}=\mathcal{F}_{d_1}\times \cdots \times \mathcal{F}_{d_m}$.}
\end{definition}

{Then we state \textcolor{black}{an analogue} of Lemma \ref{lem:parametric-Hilbert}:}
\begin{lemma}\label{lem:parametric-Hilbert:const}
    For any $\bm{F}=(f_1,\ldots,f_m)\in 
    Z_{\bm{o}}
    \subset V_{n,m,(d_1,\ldots,d_m)}$, we have the coefficient-wise inequality
    \begin{equation}
        \mathrm{HS}_{R/\langle f_1,\ldots,f_m\rangle}(z) = \mathrm{HS}_{R_L/\langle f_1,\ldots, f_m \rangle_{R_L}} \geq \mathrm{HS}_{R_L/\langle g_1,\ldots,g_m \rangle_{R_L}}(z),
    \end{equation}
{where $(g_1,\ldots,g_m)$ is the generic sequence \textcolor{black}{for $Z_{\bm{o}}$} defined in Definition \ref{df:genericsequence}, and where $R_L=L[X]$ with $L=K(\mathcal{C})$}. 
\end{lemma}

\begin{proof}
   {This can be proved similarly to Lemma \ref{lem:parametric-Hilbert}, see \cite[Lemma 1]{Froberg} for a proof by Fr\"{o}berg.}
\end{proof}

\begin{proposition}\label{prop:generic_equivalence2}
{
    Assume that $K$ is an infinite field.
    Then, for given $n$, $m$, and $(d_1,\ldots,d_m)$, the following conditions (1), (2), and (3) are equivalent to each other:
    \begin{enumerate}
        \item[(1)] There exists a generalized cryptographic semi-regular sequence $\bm{F}\in Z_{\bm{o}}$ \textcolor{black}{with $D:=\widetilde{d}_{\rm reg}(\langle \bm{F} \rangle)<{\infty}$} such that $\dim_K (R/\langle \bm{F}\rangle)_D=1$.
        \item[(2)] The property that $\bm{F}\in Z_{\bm{o}}$ is generalized cryptographic semi-regular \textcolor{black}{with $D=\widetilde{d}_{\rm reg}(\langle \bm{F} \rangle)<{\infty}$ satisfying $\dim_K (R/\langle \bm{F}\rangle)_D=1$} is generic on $Z_{\bm{o}}$ (in terms of Definition \ref{def:restricted-generic}).
        \item[(3)] For $\mathcal{C}$, $L$, $R_L$, and $\bm{G}=(g_1,\ldots,g_m)$ as in Definition \ref{df:genericsequence}, we have
        \begin{equation}\label{eq:Hil-generic2}
                    \mathrm{HS}_{R_L/\langle g_1,\ldots,g_m \rangle}(z) \equiv \frac{\prod_{j=1}^m(1-z^{d_j})}{(1- z)^n} \pmod{z^{D'}}
        \end{equation}
        and $\dim_L (R_L/\langle \bm{G}\rangle)_{D'}=1$ for $D'=\widetilde{d}_{\rm reg}(\langle \bm{G}\rangle)<{\infty}$.
    \end{enumerate}}
\end{proposition}

\begin{proof}
    {
    The implication (2) $\Rightarrow$ (1) is obvious.}

    {
    To see (1) $\Rightarrow$ (3), by Lemmas \ref{lem:Hil_lex} and \ref{lem:parametric-Hilbert:const}, we have the same inequalities as in \eqref{eq:ineq}.
    Since $\bm{F}$ is generalized cryptographic semi-regular (i.e., $\widetilde{d}_{\rm reg}(\langle \bm{F}\rangle)$-regular), $\bm{G}$ is also $\widetilde{d}_{\rm reg}(\langle \bm{F}\rangle)$-regular.
    Moreover, since $\mathrm{Krull.dim}(R_L/\langle \bm{G}\rangle) \geq 1$, the inequality $1=\mathrm{HF}_{R/\langle \bm{F}\rangle}(d) \geq \mathrm{HF}_{R_L/\langle \bm{G}\rangle}(d)$ becomes an equality for any $d \geq D$, so that $\mathrm{HS}_{R/\langle \bm{F}\rangle}(z)=\mathrm{HS}_{R_L/\langle \bm{G} \rangle}(z)$ and $\widetilde{d}_{\rm reg}(\langle \bm{G}\rangle)=D <{\infty}$.}

    {The implication (3) $\Rightarrow$ (2) follows from Theorem \ref{thm:parametric}.}
\end{proof}

{
From Proposition \ref{prop:generic_equivalence2} and Lemma \ref{lem:other_point}, if there exists a generalized cryptographic semi-regular sequence $\bm{F}\in V_{n,m,(d_1,\ldots,d_m)}$ with $\dim_K(R/\langle \bm{F}\rangle)_D=1$ for $D=\widetilde{d}_{\rm reg}(\langle \bm{F} \rangle)<{\infty}$, it follows that $Z_{\bm{p}}\cap \Lambda^{\rm gcsr}_{n,m,(d_1,\ldots,d_m)}$ includes a non-empty Zariski-open set in $Z_{\bm{p}}$ equipped with the induced topology for any $\bm{p}\in \mathbb{P}^{n-1}(K)$, which implies the following conjecture:}

\begin{conjecture}[cf.\ {\cite[Conjecture 4.3.4]{KY24b}}]\label{conj:A}
    {Assume that $K$ is an infinite field.
    For {any} $n$, $m$, and $(d_1,\ldots,d_m)$ and for any $\bm{p}\in \mathbb{P}^{n-1}(K)$, the property that $\bm{F}\in V_{n,m,(d_1,\ldots,d_m)}$ is generalized cryptographic semi-regular is generic on $Z_{\bm{p}}$.}
\end{conjecture}

{Let us raise a more general conjecture:}

\begin{conjecture}\label{conj:B}
    {Assume that $K$ is an infinite field.
    For {any} $n$, $m$, and $(d_1,\ldots,d_m)$, the property that $\bm{F}\in V_{n,m,(d_1,\ldots,d_m)}$ is generalized cryptographic semi-regular is generic on $V_{n,m,(d_1,\ldots,d_m)}^{\dim \geq 1}$.}
\end{conjecture}



\subsection{Genericness on \textcolor{black}{cryptographic} semi-regularity and 
weakly reverse lexicographicness}

In this subsection, let $K$ be an infinite field.
Here we give a proof of Theorem \ref{th:semi+wrl} and Theorem \ref{th:gsemi+wrl} using the notion of ``parametric Gr\"obner basis'' for the ideal generated by a ``generic sequence'' $\bm{G}=(g_1,\ldots,g_m)$ under the setting in Section \ref{subsec:generic} and Definition \ref{df:genericsequence}. 
We fix a monomial ordering {$\prec$} on {the set $\mathcal{T}$ of monomials in $X=\{x_1,\ldots,x_n\}$} 
as the {graded reverse lexicographical ordering} with $x_1\succ x_2\succ \cdots \succ x_n$. 
As additional notations, {we simply write ${\rm LM(*)}$ and ${\rm in(*)}$ for ${\rm LM}_{\prec}(*)$ and ${\rm in}_{\prec}(*)$, respectively,} and {we} define the following:

For each element \textcolor{black}{$\bm{a}\in\mathcal{A}_{(d_1,\ldots,d_m)}:=\mathbb{A}^N(K)=\prod_{j=1}^m \mathbb{A}^{N_j}(K)$ with $N_j=\binom{n+d_j-1}{d_j}$}, we define a map $\pi_{\bm{a}}$ from $K[\mathcal{C}][X]$ to $K[X]$ by substituting each $c_{j,t}$ with $\bm{a}_{j,t}$. 
The map $\pi_{\bm{a}}$ is generalized to a map from $K[\mathcal{C}][X]^m$ to $K[X]^m$, for which we use the same symbol.
{Thus, for a sequence $\bm{F}$ in $V_{n,m,(d_1,\ldots,d_m)}$, there exists $\bm{a}$ in $\mathcal{A}$ uniquely 
such that $\bm{F}=\pi_{\bm{a}}(\bm{G})$, \textcolor{black}{where $\bm{G}$ is a generic sequence in terms of Fr\"{o}berg, see the second paragraph of Section \ref{subsec:generic} for the definition}.
We also note that, for a polynomial $h$ in $L[X]$ {with $L=K(\mathcal{C})$}, if $\pi_{\bm{a}}$ does not annihilate any of {the} denominators of {its} coefficients, $\pi_{\bm{a}}(h)$ can be defined.}
%
Then, for each {sequence $\bm{F}$ in $V_{n,m,(d_1,\ldots,d_m)}$}, let ${\rm GB}(\bm{F})$ denote the reduced Gr\"obner basis of the ideal $\langle \bm{F}\rangle_R$ with respect to $\prec$.  
Moreover, for the generic sequence $\bm{G}$, let ${\rm GB}(\bm{G})$ denote the reduced Gr\"obner basis of the ideal $\langle \bm{G}\rangle_{L_0[X]}$ with respect to $\prec$. 
Then, {it follows from} $\bm{G}\subset L_0[X]\subset L[X]$ {that} ${\rm GB}(\bm{G})$ is also the reduced Gr\"obner basis of the ideal 
$\langle \bm{G}\rangle_{L[X]}$ with respect to $\prec$.
{Therefore, in the following, for deciding the shape of ${\rm GB}(\bm{G})$, 
we assume that all elements in $\mathcal{C}$ are algebraically independent over $K$. 
(See Remark \ref{rem:Diem}.)}
{Recall that} ${\rm LM}({\rm GB}(\bm{F}))$ forms the {minimal} generating set of 
the initial ideal {${\rm in}(\langle \bm{F}\rangle_R):=\langle {\rm LM}(\langle \bm{F}\rangle_R)\rangle_R$}. 
We simply write $\langle \bm{G}\rangle$ and $\langle \bm{F}\rangle$ 
for $\langle \bm{G}\rangle_{L[X]}$ and $\langle \bm{F}\rangle_{K[X]}$, 
respectively. 

\begin{theorem}\label{thm:semi+wrl}
If there exists a cryptographic semi-regular sequence $\bm{F}$ in \\
$V_{n,m,(d_1,\ldots,d_m)}$, that is, $\bm{F}=\pi_{\bm{a}}(\bm{G})$ 
{for some element $\bm{a}\in \mathcal{A}_{(d_1,\ldots,d_m)}$,}
such that {${\rm in}(\langle \bm{F}\rangle)$} 
is weakly reverse lexicographic, 
then we have ${\rm GB}(\bm{F})=\pi_{\bm{a}}({\rm GB}(\bm{G}))$ {and ${\rm LM}({\rm GB}(\bm{F}))={\rm LM}({\rm GB}(\bm{G}))$. 
Thus, ${\rm in}(\langle \bm{G}\rangle)$ is also weakly reverse lexicographic. 
Then Theorem \ref{thm:parametric} can be applied directly to $\bm{G}$, from which 
it follows that 
there exists a non-empty Zariski-open set $U_{\mathcal{A}}$ of \textcolor{black}{$\mathcal{A}_{(d_1,\ldots,d_m)}$} 
such that  
${\rm LM}({\rm GB}(\bm{F}))={\rm LM}({\rm GB}(\pi_{\bm{b}}(\bm{G})))$ and ${\rm in}(\langle \pi_{\bm{b}}(\bm{G})\rangle)$ is weakly reverse lexicographic for any $\bm{b}\in U_{\mathcal{A}}$}. 
\end{theorem}

{We remark that,} for each {homogeneous} ideal $I$, since the 
set of the leading monomials of the reduced Gr\"obner basis of $I$ forms the minimal generating 
set of the initial ideal {${\rm in}(I)$}, 
the property of ``weakly reverse lexicographicness'' depends on the leading monomials of the 
reduced Gr\"obner basis. 
{Considering the subset $U$ of $V_{n,m,(d_1,\ldots,d_m)}$ corresponding to $U_{\mathcal{A}}$ in Theorem \ref{thm:semi+wrl}, 
we have the following:}

\begin{corollary}\label{cor:semi+wrl}
If there exists a cryptographic semi-regular sequence    
$\bm{F}$ in $V_{n,m,(d_1,\ldots,d_m)}$ 
such that {${\rm in}(\langle \bm{F}\rangle)$} 
is weakly reverse lexicographic, 
then there also exists a non-empty Zariski-open set $U$ of $V_{n,m,(d_1,\ldots,d_m)}$ 
such that {
$\bm{H}$ is cryptographic \textcolor{black}{semi-regular} and ${\rm in}(\langle \bm{H}\rangle)$}
is weakly reverse lexicographic for any $\bm{H}$ in $U$. 
{Namely, Fr\"{o}berg's conjecture and Moreno-Socias' one are both true for $(n,m,d_1,\ldots,d_m)$.}
\end{corollary}

\noindent
{{\it Proof of Theorem \ref{thm:semi+wrl}}.\ 
It suffices to show the following 
by induction argument on a {non-negative} integer $d$:}
\begin{equation}\label{eq:GBF=GBG}
{\rm GB}(\bm{F})_{d}=\pi_{\bm{a}}({\rm GB}(\bm{G})_{d}) \mbox{ with } 
{\rm LM}({\rm GB}(\bm{F})_{d})={\rm LM}({\rm GB}(\bm{G})_{d}).
\end{equation}

{For $d=0$, it is clear as ${\rm GB}(\bm{F})_0={\rm GB}(\bm{G})_0=\emptyset$. 
Thus, we show that the \textcolor{black}{condition} \eqref{eq:GBF=GBG} holds for $d>0$} under the {\em assumption on induction} 
that it holds for $d-1$. 
{Put} ${\rm GB}(\bm{F})_d=\{h_1,\ldots,h_v\}$ and ${\rm GB}(\bm{G})_d=\{\bar{h}_1,\ldots,\bar{h}_{\bar{v}}\}$ {with $v=\# ({\rm GB}(\bm{F})_d)$ and $\bar{v}=\# ({\rm GB}(\bm{G})_d)$}, 
where we {order $h_1,\ldots,h_v$ and $\bar{h}_1,\ldots,\bar{h}_{\bar{v}}$} 
{so that ${\rm LM}(h_1)\succ {\rm LM}(h_2)\succ \cdots \succ {\rm LM}(h_v)$ and 
${\rm LM}(\bar{h}_1)\succ {\rm LM}(\bar{h}_2)\succ \cdots \succ {\rm LM}(\bar{h}_{\bar{v}})$, 
respectively.} 
By the {induction hypothesis}, we have 
{${\rm LM}(\mathrm{GB}(\bm{F})_{<d})={\rm LM}(\mathrm{GB}(\bm{G})_{<d})$} and so 
${\rm LM}(\langle \bm{F}\rangle_{<d})={\rm LM}(\langle \bm{G}\rangle_{<d})$. 
We let $Y:={\rm LM}(\langle \bm{F}\rangle_{<d})$. 
{In the following we show $v=\bar{v}$ and $h_i=\pi_{\bm{a}}(\bar{h}_i)$ with 
${\rm LM}(h_i)={\rm LM}(\bar{h}_i)$ for each $i$.}

\medskip
\noindent
\underline{\bf Claim 1}: $\bar{v}=v$ holds. 

Let $W:=\{ts: s\in Y,\ t\in \mathcal{T}_{d-\deg(s)}\}$. 
By {the} theory of Gr\"obner basis, the following holds: 
\[
{\rm LM}(\langle \bm{G}\rangle_d) = W\cup {\rm LM}({\rm GB}(\bm{G})_d), \quad 
{\rm LM}(\langle \bm{F}\rangle_d) = W\cup {\rm LM}({\rm GB}(\bm{F})_d)
\]
Moreover, {${\rm LM}(\langle \bm{G}\rangle_d) $ (resp.\ ${\rm LM}(\langle \bm{F}\rangle_d) $) \textcolor{black}{forms a linear basis} of ${\rm in}(\langle \bm{G}\rangle)_d$ (resp.\ ${\rm in}(\langle \bm{F}\rangle)_d$), respectively.}
Here we note {${\rm HS}_{K[X]/\langle \bm{F}\rangle}(z)={\rm HS}_{K[X]/{\rm in}(\langle \bm{F}\rangle)}(z)$ and ${\rm HS}_{L[X]/\langle \bm{G}\rangle}(z)={\rm HS}_{L[X]/{\rm in}(\langle \bm{G}\rangle)}(z)$}.

On the other hand, as $\bm{F}$ is cryptographic semi-regular, {so is $\bm{G}$ by (1) $\Rightarrow$ (4) of Proposition \ref{prop:generic_equivalence}, whence}
${\rm HS}_{K[X]/\langle \bm{F}\rangle}(z)={\rm HS}_{L[X]/\langle \bm{G}\rangle}(z)$. 
Therefore, their $d$-th coefficients, {namely} $\dim_K K[X]_d/\langle \bm{F}\rangle_d$ and $\dim_L L[X]_d/\langle \bm{G}\rangle_d$, 
coincide {with each other}, from which we have 
$v=\# ({\rm GB}(\bm{F})_d)=\# ({\rm GB}(\bm{G})_d)=\bar{v}$. 

\medskip
\noindent
\underline{\bf Claim 2}: 
For each $i$, {there is an element $\bar{\ell}_i$ in $\langle \bm{G}\rangle\cap K[\mathcal{C}][X]$ {with} ${\rm supp}(\bar{\ell}_i)\cap W=\emptyset$ 
{such that} $\pi_{\bm{a}}(\bar{\ell}_i)$ is defined 
{and is} a non-zero constant multiple of $h_i$. 
Here we denote {by ${\rm supp}(h_i)$} the support of $h_i$, that is, the set of monomials appearing in $h_i$ with non-zero coefficient.

For each $i$, as $h_i\in \langle \bm{F}\rangle$, there are polynomials $q_{i,1},\ldots,q_{i,m}\in K[X]$ such that 
\[
h_i=\sum_{j=1}^m q_{i,j}f_j=\sum_{j=1}^m q_{i,j}\pi_{\bm{a}}(g_j)
=\pi_{\bm{a}}\left(\sum_{j=1}^m q_{i,j}g_j\right). 
\]
Let {$\ell_i:=\sum_{j=1}^m q_{i,j}g_j$} 
for each $i$.
Then $\ell_i$ belongs to $K[\mathcal{C}][X]\cap \langle \bm{G}\rangle$ as $g_i\in K[\mathcal{C}][X]$. 
Moreover, as ${\rm LM}(\pi_{\bm{a}}(\ell_i))={\rm LM}(h_i)$, 
we have ${\rm LM}(\ell_i)\succeq {\rm LM}(h_i)$. 
{There are many choices on $q_{i,j}$, and $\ell_i$ may have some monomial in $W$ with non-zero coefficient. 
Thus, to eliminate such a monomial in $W$ from $\ell_i$, 
we apply the following procedure to ${\ell}_i$, {which outputs {a desired} $\tilde{\ell}_i$}: }

\medskip
\noindent
{(Step 1)  Check if $\ell_i$ has no monomial in $W$, that is, ${\rm supp}(\ell_i)\cap W=\emptyset$ or not:}

\medskip
\noindent
{(Step 2: Case where ${\rm supp}(\ell_i)\cap W=\emptyset$)}\\
{We set $\tilde{\ell}_i:=\ell_i$, and quite the procedure. }

\medskip
\noindent
{
(Step 3: Case where ${\rm supp}(\ell_i)\cap W\not=\emptyset$)\\
Let $t$ be the greatest monomial in ${\rm supp}(\ell_i)\cap W$ with respect to $\prec$.
Then we compute $\ell'_{i}$ from $\ell_i$ by the following monomial \textcolor{black}{reduction,} by which 
the term with the monomial $t$ is eliminated:
\[
\ell'_{i}=\ell_{i}-b_{t}(\ell_{i})\frac{t}{{\rm LM}(\bar{p})} \bar{p}, 
\]
where $b_{t}(\ell_i)$ is the coefficient of $t$ in $\ell_i$ and 
$t$ is divided by ${\rm LM}(\bar{p})$ of some $\bar{p}\in {\rm GB}(\bm{G})_{<d}$. 
As $\pi_{\bm{a}}(\ell_i)$ is defined, }{$\pi_{\bm{a}}(b_t(\ell_i))$ is also defined.}
{Moreover, by the assumption of induction, $\pi_{\bm{a}}(\bar{p})$ is defined 
and ${\rm LM}(\pi_{\bm{a}}(\bar{p}))={\rm LM}(\bar{p})$.}
Therefore, $\pi_{\bm{a}}(\ell'_i)$ is defined and 
we have the following image of the monomial reduction:
\[
\pi_{\bm{a}}(\ell'_{i})
=  \pi_{\bm{a}}(\ell_{i})-
\pi_{\bm{a}}(b_{t}(\ell_{i}))\frac{t}{{\rm LM}(\bar{p})} \pi_{\bm{a}}(\bar{p})
 =  h_i-\pi_{\bm{a}}(b_{t}(\ell_{i}))\frac{t}{{\rm LM}(\pi_{\bm{a}}(\bar{p}))} \pi_{\bm{a}}(\bar{p})
\]
As $h_i$ is an element in {${\rm GB}(\bm{F})_d$}, we have ${\rm supp}(\pi_{\bm{a}}(\ell_i))\cap W=
{\rm supp}(h_i)\cap W=\emptyset$, 
from which $\pi_{\bm{a}}(b_{t}(\ell_i))=0$ and $\pi_{\bm{a}}(\ell'_{i})=h_i$. 
In this case, if ${\rm supp}(\ell'_{i})\cap W$ is still not empty, 
its greatest monomial is smaller than $t$. 
Replacing $\ell_i$ with $\ell'_{i}$, we go back to Step 1.

\medskip
{
Since $W$ is a finite set and the greatest monomial of ${\rm supp}(\ell_i)\cap W$ is decreasing, 
the procedure terminates in finitely many steps. }
{Then, we convert the output polynomial $\tilde{\ell}_i$ 
to a polynomial $\bar{\ell}_i$ in $K[\mathcal{C}][X]$ 
by multiplying the least common multiple (LCM) of the denominators of coefficients of 
$\tilde{\ell}_i$ to itself. 
Then $\bar{\ell}_i$ satisfies the condition of Claim 2.
}

\medskip
\noindent
\underline{\bf Claim 3}: For each $i$, {the image} $\pi_{\bm{a}}(\bar{h}_i)$ is defined and 
it coincides with $h_i$. 

Let {${\rm vect}(\bar{\ell}_i)$ be the {\em coefficient vector} of $\bar{\ell}_i$}, where components 
are indexed by monomials in $\mathcal{T}_d$ in decreasing order. 
As {$\bar{\ell}_i$} does not have any monomial in $W$, 
we consider a sub-vector {${\rm vect}_0(\bar{\ell}_i)$ which is obtained from ${\rm vect}(\bar{\ell}_i)$} 
by eliminating all coefficients of monomials in $W$. 
Let $t_1\succ t_2\succ \cdots \succ t_w$ are indices of vectors. 
{(Thus, $\{t_1,\ldots,t_w\}=\mathcal{T}_d\smallsetminus W$.)}
Since ${\rm in}(\langle \bm{F}\rangle)$ 
is weakly reverse lexicographic, 
{
it can be shown easily that $t_1={\rm LM}(h_1)$, \ldots, $t_v={\rm LM}(h_v)$. 
Because,
any monomial of {degree $d$} not smaller than ${\rm LM}(h_v)$ should belong to $W\cup \{{\rm LM}(h_1),\ldots,{\rm LM}(h_{v})\}$, 
which implies that ${\rm LM}(h_1),\ldots,{\rm LM}(h_v)$ are top $v$ monomials in $\mathcal{T}_d\smallsetminus W$.}
Letting {$\bar{\ell}_i=a_{i,t_1}t_1+\cdots +a_{i,t_v}t_v+a_{i,t_{v+1}}t_{v+1}+\cdots 
+a_{i,t_w}t_w$} {with $a_{i,t_j}\in K[\mathcal{C}]$} for each $i$,
a sub-Macaulay matrix $M$ over $K[\mathcal{C}]$ 
is constructed from vectors {${\rm vect}_0(\bar{\ell}_1),\ldots, 
{\rm vect}_0(\bar{\ell}_v)$} as follows:
\[
M:=\begin{pmatrix}
a_{1,t_1} & \ldots & a_{1,t_v} & a_{1,t_{v+1}} & \ldots & a_{1,t_w}\\
\vdots  & \ddots & \vdots  &  \vdots      &           & \vdots\\
a_{v,t_1} & \ldots & a_{v,t_v} & a_{v,t_{v+1}} & \ldots & a_{v,t_w}
\end{pmatrix}
\]
{Then, we apply $\pi_{\bm{a}}$ naturally to $M$ by $\pi_{\bm{a}}(M)_{i,t}=\pi_{\bm{a}}(M_{i,t})$ for each $(i,t)$-entry.
As $\pi_{\bm{a}}(\bar{\ell}_i)$ is a non-zero constant 
multiple of $h_i$,}
we have 
\[
\pi_{\bm{a}}(M)=\begin{pmatrix}
b_{1,t_1} & 0 & \ldots & 0 & b_{1,t_{v+1}} & \ldots & b_{1,t_w}\\
0        & {b_{2,t_2}} & \ldots & 0 & {b_{2,t_{v+1}}} & \ldots & {b_{2,t_w}}\\
\vdots   & \vdots & \ddots  & \vdots &  \vdots      &           & \vdots\\
0        & 0  & \ldots & b_{v,t_v} & b_{v,t_{v+1}} & \ldots & b_{v,t_w}
\end{pmatrix}, 
\]
where $b_{i,t_j}=\pi_{\bm{a}}(a_{i,t_j})$ for each $i$ and $j$, and $b_{i,t_i}\not=0$ for $1\leq i\leq v$. 
Let $M_0$ be the square matrix consisting of the first $v$ left columns of $M$ and 
$\overline{M_0}$ the adjugate matrix of $M_0$.
Namely we have
\[
M_0:=\begin{pmatrix}
a_{1,t_1} & \ldots & a_{1,t_v} \\
\vdots  & \ddots & \vdots  \\
a_{v,t_1} & \ldots & a_{v,t_v}
\end{pmatrix},
\]
and $\overline{M_0}$ is a matrix over $K[\mathcal{C}]$. Then, 
$\pi_{\bm{a}}(M_0)$ is a diagonal matrix with non-zero diagonal entries \textcolor{black}{$b_{1,t_1},\ldots,b_{v,t_v}$} 
and $\det(\pi_{\bm{a}}(M_0))=\pi_{\bm{a}}(\det(M_0))\not=0$. 
Hence we have $\det(M_0)\not=0$ as an element in $K[\mathcal{C}]$. 
Since $\overline{M_0}M_0=\det(M_0)E_v$,
where $E_v$ denotes the identity matrix of degree $v$,
we have 
\[
\textcolor{black}{
\overline{M_0}M=(\det(M_0)E_v, \overline{M_0}V_{v+1},\ldots,\overline{M_0}V_w), }
\]
where $V_j$ denotes the $j$-th column of $M$ for each $j$. 

For each $i$, let $\bar{h}'_i$ be a polynomial in $K[\mathcal{C}][X]$ 
converted from the $i$-th low of $\overline{M_0}M_0$. 
Then, 
{since each $\bar{\ell}_i$ belongs to $\langle \bm{G}\rangle \cap 
K[\mathcal{C}][X]$ and the left-multiplication of $\overline{M}_0$ induces low-operations on $M$ over $K[\mathcal{C}]$}, 
each $\bar{h}'_i$ belongs to $\langle \bm{G}\rangle\cap K[\mathcal{C}][X]$. 
{Moreover, we have $t_i={\rm LM}(\bar{h}'_i)\in {\rm in}(\langle \bm{G}\rangle)$} 
and ${\rm LC}(\bar{h}'_i)=\det(M_0)$. 
{Since $t_i$ cannot be divided by any monomial in ${\rm LM}({\rm GB}(\bm{F})_{<d})={\rm LM}({\rm GB}(\bm{G}))_{<d})$, 
the monomials $t_1,\ldots,t_v$ should belong to the minimal generating set ${\rm LM}({\rm GB}(\bm{G}))$ of 
${\rm in}(\langle \bm{G}\rangle)$.
Comparing the number $v=\# {\rm LM}(\bm{G}_d)$, 
we have $\{t_1,\ldots,t_v\}={\rm LM}(\bm{G}_d)$.} 

{
Now let $\hat{h}_i=\frac{\bar{h}'_i}
{{\rm LC}(\bar{h}'_i)}=\frac{\bar{h}'_i}{\det(M_0)}$ for each $i$. 
Then, we have ${\rm LC}(\hat{h}_i)=1$ for each $i$. 
Since 
${\rm GB}(\bm{G})_{<d}\cup \{\hat{h}_1,
\ldots,\hat{h}_v\}$ forms a set of 
inter-reduced elements, 
\textcolor{black}{$\{\hat{h}_1,\ldots,\hat{h}_v\}$} 
coincides with ${\rm GB}(\bm{G})_d$, that is,
$\bar{h}_i=\hat{h}_i$ for each $i$.} 


Finally we show $\pi_{\bm{a}}(\bar{h}_i)=h_i$ for each $i$. 
Since $\pi_{\bm{a}}(\det(M_0))\not=0$, the image $\pi_{\bm{a}}(\bar{h}_i)$ is defined for each $i$.
Moreover, ${\rm LM}(\pi_{\bm{a}}(\bar{h}_i))=t_i={\rm LM}(h_i)$, {${\rm LC}(\pi_{\bm{a}}(\bar{h}_i))=1$, 
and $\bar{h}_i$ is {\em reduced} with respect to 
${\rm GB}(\bm{G})_{<d}$ and $t_j$ for $j\not=i$. }
This implies $\pi_{\bm{a}}(\bar{h}_i)=h_i$ for each $i$ 
{by the uniqueness of the reduced Gr\"obner basis with respect to a fixed monomial ordering.}
\qed

%
\medskip
{Finally we consider a generalized cryptographic semi-regular sequence belonging to $Z_{\bm{o}}$, \textcolor{black}{where ${Z}_{\bm{o}}$ is defined in Section \ref{subsec:rest_generic} (see \eqref{eq:Zp})}.
In this case, the set $\mathcal{C}$ of parameters is 
$\{c_{j,t}:1\leq j\leq m,\ t\in \mathcal{T}_{d_j}\smallsetminus \{x_n^{d_j}\}\}$. 
Moreover, we set 
$\mathcal{B}_{d}:=\{(a_{t})_{t\in \mathcal{T}_d\smallsetminus \{x_n^d\}}:
a_t\in K\}$ for each positive integer $d$, 
and set $\mathcal{B}_{(d_1,\ldots,d_m)}:=\prod_{j=1}^m \mathcal{B}_{d_j}$. 
Then, for each element $\bm{a}$ in $\mathcal{B}_{(d_1,\ldots,d_m)}$, 
we define a map $\pi_{\bm{a}}$ from $K[\mathcal{C}][X]$ to $K[X]$ 
by substituting each $c_{j,t}$ with $a_{j,t}$.
The map $\pi_{\bm{a}}$ 
is generalized to {a map from $K[\mathcal{C}][X]^m$ to $K[X]^m$}, 
for which we use the same symbol. 
\textcolor{black}{Let $\bm{G}$ be a generic sequence for $Z_{\bm{o}}$ defined in Definition \ref{df:genericsequence}.}
Under this setting, we can apply the same proof in Theorem \ref{thm:semi+wrl} 
to obtain the following {results} corresponding to 
Theorem \ref{thm:semi+wrl} and Corollary \ref{cor:semi+wrl}. }

\begin{proposition}\label{prop:gsemi+wrl}
\textcolor{black}{Let $\bm{G}$ be a generic sequence for $Z_{\bm{o}}$.}
If there exists a generalized cryptographic semi-regular sequence $\bm{F}$ 
in $Z_{\bm{o}}$, 
that is, $\bm{F}=\pi_{\bm{a}}(\bm{G})$ for some element $\bm{a}\in \mathcal{B}_{(d_1,\ldots,d_m)}$, 
such that $\dim_K (K[X]/\langle \bm{F}\rangle)_D=1$ for $D=\widetilde{d}_{\rm reg}(\langle \bm{F}\rangle)<\infty$ and 
${\rm in}(\langle \bm{F}\rangle)$ is weakly reverse lexicographic, then 
${\rm GB}(\bm{F})=\pi_{\bm{a}}({\rm GB}(\bm{G}))$ 
{and ${\rm LM}({\rm GB}(\bm{F}))={\rm LM}({\rm GB}(\bm{G}))$. 
In this case, $\bm{G}$ is generalized cryptographic semi-regular 
with \textcolor{black}{$\dim_L (L[X]/\langle \bm{G}\rangle)_D=1$} for $D=\widetilde{d}_{\rm reg}(\langle \bm{G}\rangle)<\infty$ 
and ${\rm in}(\langle \bm{G}\rangle)$ is weakly reverse lexicographic.} 
{Moreover, in this case, 
there exists a non-empty Zariski-open set $U$ of $Z_{\bm{o}}$ 
{with respect to the induced topology} such that 
$\bm{H}$ is generalized cryptographic semi-regular and 
${\rm in}(\langle \bm{H}\rangle)$ is weakly reverse lexicographic for any $\bm{H}$ in $U$. }
\end{proposition}

\begin{proof}
By Proposition \ref{prop:generic_equivalence2}, we have 
\[{\rm HS}_{K[X]/\langle \bm{F}\rangle}(z) \equiv 
{\rm HS}_{L[X]/\langle \bm{G}\rangle}(z)\equiv 
\frac{\prod_{j=1}^m(1-z^{d_j})}{(1-z)^n} \pmod{z^D},
\]
where $D=\widetilde{d}_{\rm reg}(\langle \bm{F}\rangle)=\widetilde{d}_{\rm reg}(\langle \bm{G}\rangle)$ 
{and \textcolor{black}{$\dim_L (L[X]/\langle \bm{G}\rangle)_D=1$}.
In this case, as $\bm{F}$ and $\bm{G}$ have $(0:\cdots:0:1)$ as their projective zero, 
$x_n^D$ cannot belong neither to $\langle \bm{F}\rangle_D$ nor to $\langle \bm{G}\rangle_D$. 
This implies that 
\textcolor{black}{$(K[X]/\langle \bm{F}\rangle)_{D}$ and $(L[X]/\langle \bm{G}\rangle)_D$} have a common linear 
basis $\{x_n^D\}$, 
which also implies that 
${\rm max.GB.deg}(\langle \bm{F}\rangle):={\rm max}\{\deg(h):h\in {\rm GB}(\bm{F})\}$ 
and ${\rm max.GB.deg}(\langle \bm{G}\rangle)$ are {both} smaller than $D$.}
Thus, all arguments in the proof of Theorem \ref{thm:semi+wrl} can 
be applied to the case of Proposition \ref{prop:gsemi+wrl}, 
from which we obtain ${\rm GB}(\bm{F})=\pi_{\bm{a}}({\rm GB}(\bm{G}))$ 
{and ${\rm LM}({\rm GB}(\bm{F}))={\rm LM}({\rm GB}(\bm{G}))$.}

{Then, by applying Theorem \ref{thm:parametric} to $\bm{G}$, 
there exists a non-empty Zariski-open set $U_{\mathcal{B}}$ of $\mathcal{B}_{(d_1,\ldots,d_m)}$ 
such that {${\rm LM}({\rm GB}(\bm{F}))={\rm LM}({\rm GB}(\pi_{\bm{b}}(\bm{G})))$} 
for any $\bm{b}$ in $U_B$, that is, \textcolor{black}{${\rm in}(\langle \bm{F}\rangle)=
{\rm in}(\langle \pi_{\bm{b}}(\bm{G})\rangle))$}. 
This implies that the subset $U$ of $Z_{\bm{o}}$ corresponding to $U_{\mathcal{B}}$ 
is also a non-empty Zariski-open set such that $\bm{H}$ 
is generalized cryptographic semi-regular and 
${\rm in}(\langle \bm{H}\rangle)$ is weakly reverse lexicographic for any $\bm{H}\in U_{\mathcal{B}}$.}
\end{proof}

\begin{proposition}\label{prop:Z_p}
{Suppose that the conditions of Proposition \ref{prop:gsemi+wrl} are satisfied.
Then, for any $\bm{p}\in \mathbb{P}^{n-1}(K)$, 
there exists a non-empty Zariski-open set $U_{\bm{p}}$ of $Z_{\bm{p}}$ 
with respect to the induced topology such that 
$\bm{H}$ is generalized cryptographic semi-regular and 
${\rm in}(\langle \bm{H}\rangle)$ is weakly reverse lexicographic for any $\bm{H}$ in $U_{\bm{p}}$. }
\end{proposition}
\begin{proof}
%
We show {the assertion} by using the arguments in the proof of 
Lemma \ref{lem:other_point}
Let $\bm{p}:=(p_1:\ldots:p_n)$ be {a point} in $\mathbb{P}^{n-1}(K)$. 
Without loss of generality, we may assume that $p_n\not=0$ and moreover $p_n=1$, 
and consider the 
following maps, where 
$\mathcal{V}:=V_{n,m,(d_1,\ldots,d_m)}$: 
\begin{eqnarray*}
&& \phi_{\bm{p}}: 
(a_1:\cdots: a_{n-1}:a_n)\mapsto 
(a_1-p_1a_n:\cdots:a_{n-1}-p_{n-1}a_n:a_n),\\
&& \phi_{\bm{p}}^*: K[X]\ni h(x_1,x_2,\ldots,x_n) \mapsto 
h(x_1-p_1x_n,\ldots ,x_{n-1}-p_{n-1} x_n,x_n)\in K[X],\\
&& \Phi_{\bm{p}}:\mathcal{V} \in (h_1,\ldots,h_m)\mapsto 
(\phi_{\bm{p}}^*(h_1),\ldots,\phi_{\bm{p}}^*(h_m))\in \mathcal{V}.
\end{eqnarray*}
Then 
$\phi_{\bm{p}}^*$ is a ring homomorphism 
and $\Phi_{\bm{p}}$ is an invertible linear map on $\mathcal{V}$ and 
its restriction on $Z_{\bm{o}}$ induces a bijection between $Z_{\bm{o}}$ and $Z_{\bm{p}}$.
We also 
let $\prec$ be the graded reverse lexicographical ordering with $x_1\succ x_2\succ \cdots \succ x_n$.

Now \textcolor{black}{we} suppose that the conditions of Proposition \ref{prop:gsemi+wrl} hold. 
Then \textcolor{black}{the generic sequence $\bm{G}$ for $Z_{\bm{o}}$} is generalized cryptographic semi-regular sequence with \textcolor{black}{$\dim_L (L[X]/\langle \bm{G}\rangle)_D=1$} for $D=\widetilde{d}_{\rm reg}(\langle \bm{G}\rangle)<\infty$ and ${\rm in}(\langle \bm{G}\rangle)$ is 
weakly reverse lexicographic, and there is a non-empty Zariski-open set $U_{\bm{o}}$ in $Z_{\bm{o}}$ 
such that, for any $\bm{H}$ in $U_{\bm{o}}$, $\bm{H}$ is generalized cryptographic semi-regular and 
{${\rm in}(\langle \bm{H}\rangle)$} is weakly reverse lexicographic. 

Meanwhile, since $\phi_{\bm{p}}^*$ does not change LM for any polynomial in $L[X]$ and 
${\rm in}(\langle \bm{F}\rangle)={\rm in}(\langle \Phi_{\bm{p}}(\bm{F})\rangle)$ for any $\bm{F}$ in $U_{\bm{o}}$,
it can be shown easily that, for any $\bm{H}$ in $U_{\bm{o}}$, 
$\Phi_{\bm{p}}(\bm{H})$ is generalized \textcolor{black}{cryptographic} semi-regular and 
${\rm in}(\langle \Phi_{\bm{p}}(\bm{H})\rangle)$ is weakly reverse lexicographic. 
Let $U_{\bm{p}}=\Phi_{\bm{p}}(U_{\bm{o}})$.  Then, $U_{\bm{p}}$ is also a non-empty 
Zariski-open set of $\Phi_{\bm{p}}(Z_{\bm{o}})=Z_{\bm{p}}$ \textcolor{black}{(see the proof of Lemma \ref{lem:other_point})}. 
Thus, $U_{\bm{p}}$ satisfies the assertion of Proposition \ref{prop:Z_p}. 
\end{proof}

\begin{remark}
{It is clear that the converse of Proposition \ref{prop:Z_p} also holds, that is, 
if there is a generalized cryptographic semi-regular sequence $\bm{H}$ in $Z_{\bm{p}}$ 
such that $\dim_K(K[X]/\langle \bm{H}\rangle)_D=1$ for $D=\widetilde{d}_{\rm reg}(\langle \bm{H}\rangle)<\infty$ 
and ${\rm in}(\langle \bm{H}\rangle)$ is weakly reverse lexicographic, 
then the sequence $\bm{F}=(\Phi_{\bm{p}})^{-1}(\bm{H})$ belongs to $Z_{\bm{o}}$ and 
it satisfies the condition of Proposition \ref{prop:gsemi+wrl}. }
\end{remark}

\subsubsection*{Acknowledgement}
This work was supported by JSPS Grant-in-Aid for Young Scientists 23K12949, JSPS Grant-in-Aid for Scientific Research (C) 21K03377, and JST CREST Grant Number JPMJCR2113.


\begin{thebibliography}{99}



\bibitem{Bar}
M.\ Bardet:
\'Etude des syst\'emes alg\'ebriques surd\'etermin\'es. 
Applications aux codes correcteurs et \'a la cryptographie. 
PhD thesis, Universit\'e Paris IV, 2004.

\bibitem{BFS}
M.\ Bardet, J.-C.\ Faug\`{e}re, and B.\ Salvy:
On the complexity of Gr\"{o}bner basis computation of semi-regular overdetermined algebraic equations (extended abstract). 
In: Proceedings of the International Conference on Polynomial System Solving, 71--74, 2004.

\bibitem{BFS-F5}
M.\ Bardet, J.-C.\ Faug\`{e}re, and B.\ Salvy:
On the complexity of the $F_5$ Gr\"{o}bner basis algorithm.
Journal of Symbolic Computation, {\bf 70}, 49--70, 2015.

\bibitem{BFSY}
M.\ Bardet, J.-C.\ Faug\`{e}re, B.\ Salvy, and B.-Y.\ Yang:
Asymptotic \textcolor{black}{behavior} of the degree of regularity of semi-regular polynomial systems.
In: Proceedings of Eighth International Symposium on Effective Methods in Algebraic Geometry (MEGA 2005), 2005.

\bibitem{BS}
D.\ Bayer and M.\ Stillman, M.:
A criterion for detecting m-regularity.
Invent.\ Math., {\bf 87} (1), 1--11, 1987.

\bibitem{BG01}
I.\ Bermejo and P.\ Gimenez:
Computing the Castelnuovo–Mumford regularity of some subschemes of $\mathbb{P}_K^n$ using quotients of monomial ideals.
Journal of Pure and Applied Algebra, {\bf 164}, Issues 1--2, 23--33, 2001.

\bibitem{BH}
W.\ Bruns and J.\ Herzog:
Cohen-Macaulay rings.
Revised ed.\
Cambridge Studies in Advanced Mathematics,
vol.\ {\bf 39}.
Cambridge University Press, Cambridge, 1998.






\bibitem{BNDGMT21}
M.\ Bigdeli, E.\ De Negri, M.\ M.\ Dizdarevic, E.\ Gorla, R.\ Minko, and S.\ Tsakou:
Semi-Regular Sequences and Other Random Systems of Equations.
In: Women in Numbers Europe III, {\bf 24}, pp.\ 75--114, Springer, 2021.


\bibitem{Buchberger}
B.\ Buchberger:
Ein Algorithmus zum Auffinden der Basiselemente des Restklassenringes nach einem nulldimensionalen Polynomideal.
Innsbruck: Univ.\ Innsbruck, Mathematisches Institut (Diss.), 1965.


\bibitem{CG20}
A.\ Caminata and E.\ Gorla:
Solving Multivariate Polynomial Systems and an Invariant from Commutative Algebra.
In: Arithmetic of Finite Fields (Proc.\ of WAIFI 2020), LNCS, {\bf 12542}, pp.\ 3--36, Springer, 2021.

\bibitem{CG23}
A.\ Caminata and E.\ Gorla:
Solving degree, last fall degree, and related invariants.
J.\ Symb.\ Comp., {\bf 114}, 322--335, \textcolor{black}{2023}.



\bibitem{C14}
J.\ G.\ Capaverde:
Gr\"{o}bner bases: Degree bounds and generic ideals.
PhD thesis, Clemson University, 2014. 


\bibitem{GroebnerWalk}
S.\ Collart, M.\ Kalkbrener, and D.\ Mall:
Converting bases with the Groebner walk.
J.\ Symb.\ Comp., {\bf 24}, Issues 3--4, 465--469, 1997.

\bibitem{CLO}
D.\ A.\ Cox, J.\ Little, and D.\ O'Shea.
Ideals, Varieties, and Algorithms (Fourth Edition). 
Undergraduate Texts in Mathematics, Springer, NY, 2010. 


      
\bibitem{Diem2}
C.\ Diem:
Bounded regularity.
Journal of Algebra, {\bf 423}, 1143--1160, 2015.

\bibitem{Diem}
{
C.\ Diem:
The XL-algorithm and a Conjecture from Commutative Algebra.
Asiacrypt'04, LNCS, {\bf 3329}, pp.\ 323--337, \textcolor{black}{2004.}}



\bibitem{Eisen}
D.\ Eisenbud:
Commutative Algebra: With a View Toward Algebraic Geometry.
GTM, {\bf 150}, Springer, 1995.

\bibitem{Eisen2}
{
D.\ Eisenbud:
The Geometry of Syzygies: A Second Course in Algebraic Geometry and Commutative Algebra.
Springer, GTM, {\bf 229}, 2005.}

\bibitem{EF}
C.\ Eder and J.-C.\ Faug\`{e}re:
A survey on signature-based algorithms for computing Gr\"obner bases.
Journal of Symbolic Computation {\bf 80} (2017), 719--784. 

\bibitem{F4}
J.-C.\ Faug\`{e}re:
A new efficient algorithm for computing Gr\"{o}bner bases (F4).
Journal of Pure and Applied Algebra, {\bf 139} (1999), 61--88.

\bibitem{F5}
J.-C.\ Faug\`{e}re:
A new efficient algorithm for computing Gr\"{o}bner bases 
without reduction to zero (F5).
In: Proceedings of ISSAC 2002, ACM Press, (2002), pp.\ 75--82.


\bibitem{FGLM}
J.-C.\ Faug\`{e}re, P.\ Gianni, D.\ Lazard, and T.\ Mora:
Efficient Computation of Zero-dimensional Gr\"{o}bner Bases by Change of Ordering.
J.\ Symb.\ Comp., {\bf 16} (4), 329--344, 1993. 



\bibitem{Froberg}
R.\ Fr\"{o}berg:
An inequality for Hilbert series of graded algebras.
Math.\ Scand., {\bf 56} (1985), 117--144.

\bibitem{Froberg22}
R.\ Fr\"{o}berg:
Hilbert Series of Generic Ideals in Products of Projective Spaces.
Experimental Math., {\bf 31}, Issue 4, 1370--1372, 2022.

\bibitem{FH94}
{R.\ Fr\"{o}berg and J.\ Hollman:
Hilbert Series for Ideals Generated by Generic Forms.
Journal of Symbolic Computation, {\bf 17} (2), 149--157, 1994.
}

\bibitem{FL90}
{
R.\ Fr\"{o}berg and C. L\"{o}fwall:
On Hilbert series for commutative and noncommutative graded algebras.
J.\ Pure.\ Appl.\ Algebra, {\bf 76}, 33--38 (1990).}

\bibitem{FS18}{
R.\ Fr\"{o}berg and S.\ Lundqvist:
Questions and conjectures on extremal Hilbert series.
Rev.\ Union Mat.\ Argentina, {\bf 59} (2018), no.\ 2, 415--429.}

\bibitem{FK24}
H.\ Furue and M.\ Kudo:
Polynomial XL: A Variant of the XL Algorithm Using Macaulay Matrices over Polynomial Rings.
Post-Quantum Cryptography, PQCrypto 2024, Lecture Notes in Computer Science, {\bf 14772}, pp.\ 109--143, Springer, Cham, 2024.




\bibitem{Singular}
G.-M.\ Greuerl and G.\ Pfister:
A \textcolor{black}{Singular} Introduction to Commutative Algebra.
2nd Edition, Springer, 2007. 


\bibitem{Hashemi2010}
A.\ Hashemi:
Strong Noether Position and Stabilized Regularities.
Communications in Algebra, {\bf 38}, 515--533, 2010.

\bibitem{HS}
A.\ Hashemi and W.\ M.\ Seiler:
Dimension and depth dependent upper bounds in polynomial ideal theory.
Journal of Symbolic Computation, {\bf 98}, 47--64, 2020.

\bibitem{HSS18}
A.\ Hashemi, M.\ Schweinfurter, and W.\ M.\ Seiler:
Deterministic genericity for polynomial ideals.
Journal of Symbolic Computation, {\bf 86}, 20-–50, 2018.









\bibitem{KR1}
M.\ Kreuzer and L.\ Robbiano:
Computational Commutative Algebra 1. 
Springer, 2000.

\bibitem{KR}
M.\ Kreuzer and L.\ Robbiano:
Computational Commutative Algebra 2. 
Springer, 2003.

\bibitem{KY}
M.\ Kudo and K.\ Yokoyama:
On Hilbert-Poincar\'e series of affine semi-regular polynomial sequences and related Gr\"obner bases.
In:  T.\ Takagi et al.\ (eds), Mathematical Foundations for Post-Quantum Cryptography, 
Mathematics for Industry, 26 pages, Springer, to appear (arXiv:2401.07768).

\bibitem{KY24b}
M.\ Kudo and K.\ Yokoyama:
The solving degrees for computing Gr\"{o}bner bases of affine semi-regular polynomial sequences.
Accepted for presentation at Effective Methods in Algebraic Geometry (MEGA2024), arXiv:2404.03530 or eprint/2024/52.

\bibitem{KY2}
M.\ Kudo and K.\ Yokoyama:
Generalized cryptographic semi-regular sequences: A variant of Fr\"{o}berg conjecture and a simple complexity estimation for Gr\"{o}bner basis computation. arxiv:2410.23211.

\bibitem{Lazard}
D.\ Lazard:
Gr\"{o}bner bases, Gaussian elimination and resolution of systems of algebraic equations.
In: Computer algebra (London, 1983), LNCS, {\bf 162}, pp.\ 146--156, Springer, Berlin, 1983.

\bibitem{Lazard81}
D.\ Lazard:
R\'{e}solution des syst\`{e}mes d'\'{e}quations alg\'{e}briques.
Theoretical Computer Science, {\bf 15}, Issue 1, 77--110, 1981.

\bibitem{LJ}
M.\ Lejeune-Jalabert:
Chapitre 2 Solutions des syst\`{e}mes alg\'{e}briques, dans Effectivit\'{e} de calculs polynomiaux, Cours de l'institut Fourier, no.\ 19 (1984-1985), pp.\ 53--90. http://www.numdam.org/item/CIF\_1984-1985\_\_19\_\_53\_0/





\bibitem{MS}
G.\ Moreno-Soc\'{i}as:
Autour de la fonction de Hilbert-Samuel (escaliers d'id\'{e}aux polynomiaux), Th\`{e}se, \'{E}cole Polytechnique, 1991.

\bibitem{Nabe}
K.\ Nabeshima:
Generic Gr\"obner basis of a parametric ideal and its 
application to a comprehensive Gr\"obner system.
Applicable Algebra in Engineering, Communication and Computing, 
{\bf 35}, 55--70, 2024. 



\bibitem{Pardue}
K.\ Pardue:
Generic sequences of polynomials.
Journal of Algebra, {\bf 324.4}, 579--590, 2010.






\bibitem{Spa}
P.-J.\ Spaenlehauer:
Solving Multi-Homogeneous and Determinantal Systems: Algorithms, Complexity, Applications.
PhD thesis. Pierre and Marie Curie University, Paris, France, 2012.


\bibitem{Steiner24}
M.\ J.\ Steiner:
Solving Degree Bounds for Iterated Polynomial Systems.
IACR Transactions on Symmetric Cryptology, Vol.\ {\bf 2024}, No.\ 1, pp.\ 357--411.

\bibitem{Tra}
C.\ Traverso:
Hilbert functions and the Buchberger algorithm.
J.\ Symb.\ Comp., {\bf 22.4} (1996), 355--376.





\end{thebibliography}
\end{document}